\documentclass[9pt,reqno]{amsart}

\usepackage{amsmath, amsfonts, amssymb, latexsym, amsthm}
\usepackage[pagewise]{lineno}
\usepackage{amsmath, amsfonts, amssymb, latexsym}
\usepackage[numbers,sort&compress]{natbib}
\usepackage{mathrsfs}
\usepackage{pdfcomment}
\hypersetup{hidelinks,
	colorlinks=true,
	allcolors=black,
	pdfstartview=Fit,
	breaklinks=true
}

\everymath{\displaystyle}

\newtheorem{theorem}{Theorem}
\theoremstyle{plain}

\newtheorem{lemma}[theorem]{Lemma}

\newtheorem{proposition}[theorem]{Proposition}
\newtheorem{corollary}[theorem]{Corollary}
\newtheorem{remark}[theorem]{Remark}

\numberwithin{equation}{section}
\numberwithin{theorem}{section}

\newcommand{\cU}{\mathcal{U}}

\newcommand{\R}{\mathbb{R}}

\newcommand{\N}{\mathbb{N}}

\def\Om{\Omega}

\begin{document}
\title
{Observability Inequality from Measurable Sets and the Stackelberg-Nash Game Problem for Degenerate Parabolic Equations}

\author{\sffamily Yuanhang Liu$^{1}$, Weijia Wu$^{1,*}$, Donghui Yang$^1$, Can Zhang$^2$   \\
	{\sffamily\small $^1$ School of Mathematics and Statistics, Central South University, Changsha 410083, China. }\\
 {\sffamily\small $^2$ School of Mathematics and Statistics, Wuhan University, Wuhan 430072, China. }
}
	\footnotetext[1]{Corresponding author: weijiawu@yeah.net }

\email{liuyuanhang97@163.com}
\email{weijiawu@yeah.net}
\email{donghyang@outlook.com}
\email{zhangcansx@163.com}


\maketitle

\begin{abstract}
In this study, we employ the established Carleman estimates and propagation estimates of smallness from measurable sets for real analytic functions, along with the telescoping series method, to establish an observability inequality for the degenerate parabolic equation over measurable subsets in the time-space domain. As a direct application, we formulate a captivating Stackelberg-Nash game problem and provide a proof of the existence of its equilibrium. Additionally, we characterize the set of Stackelberg-Nash equilibria and delve into the analysis of a norm optimal control problem.
\end{abstract}

\pagestyle{myheadings}
\thispagestyle{plain}
\markboth{OBSERVABILITY INEQUALITY AND STACKELBERG-NASH EQUILIBRIUM}{YUANHANG LIU, WEIJIA WU AND DONGHUI YANG}

\section{Introduction}
Observability inequality is a powerful and significant tool for investigating stability and controllability problems in evolution equations. The degenerate parabolic equation, which is a common class of diffusion equations, can describe numerous physical phenomena, such as laminar flow, large ice blocks, solar radiation-climate interactions, and population genetics (for more detailed descriptions, refer to \cite{cannarsa2016global}). As a result, the control problems for such equations have gained considerable attention, with studies on controllability and observability problems for some degenerate parabolic equations in \cite{cannarsa2016global} and its extensive references.
In \cite{cannarsa2008carleman}, the authors derived the observability inequality for a kind of degenerate parabolic problems, where the control acts on an open sets, by Carleman estimates based on the choice of suitable weighted functions and Hardy-type inequalities. In \cite{fragnelli2016interior},
the author considered non-smooth general degenerate parabolic equations in non-divergence form with the control acts on an open sets and with degeneracy and singularity occurring in the interior of the spatial domain, in presence of Dirichlet or Neumann boundary conditions. In particular, they obtained the observability inequality for the associated adjoint problem by Carleman estimates. In \cite{boutaayamou2018carleman}, the authors considered a parabolic problem with degeneracy in the interior
of the spatial domain and Neumann boundary conditions  with the control acts on an open sets. In particular, they focused
on the well-posedness of the problem and on Carleman estimates for the associated
adjoint problem. As a
consequence, new observability inequalities were established. For other controllability issues related to degenerate parabolic equations, see \cite{cannarsa2005null,cannarsa2019null,du2014approximate,floridia2014approximate} and references therein.

It is worth noting that the current researches on observability inequalities for degenerate parabolic equations have primarily focused on control exerted on open sets, with little existing results considering control exerted on measurable sets. The reason for these is that in the above works, the main technique used in the arguments, Carleman inequalities, requires to construct suitable
Carleman weights: a role for functions which requires smoothness and to have the extreme values
in proper regions associated to the control region. The
construction of such functions seems to be not possible, when the control region does not an open set. In \cite{liu2017observability}, The authors employed the available Carleman estimates, propagation estimates of smallness from measurable sets for real analytic functions, and the telescoping series method to establish an observability inequality from measurable subsets in the time-space variable for the degenerate parabolic equation with the Grushin operator in certain multidimensional domains.
On the contrary, for non-degenerate parabolic equations, there have been extensive findings pertaining to control exerted on measurable sets. The authors in \cite{apraiz2014observability,phung2013observability} establish the observability
inequality of the heat equation for the measurable subsets, and show the null
controllability with controls restricted over these sets. In \cite{escauriaza2015observation}, the authors found new quantitative estimates on the space–time analyticity of solutions to linear parabolic evolutions with time-independent analytic coefficients and applied them to obtain observability inequalities for its solutions over measurable sets. 

The current paper represents an advancement in the study of observability inequality for measurable subsets within the context of degenerate parabolic equations. However, unlike the approach discussed earlier, which relied on Carleman estimates, this paper builds upon the conventional observability inequality established in \cite{cannarsa2008carleman} and \cite{buffe2018spectral} for open subsets. Notably, it is currently not feasible to derive an observability inequality for degenerate parabolic equations from measurable subsets using the Carleman estimates method. Nevertheless, we have made an intriguing observation: When the observation domain is sufficiently distant from the singularity, we can derive an observability inequality with a constant term on the right-hand side, expressed as $e^{\frac{C}{T^k}}$, where $C$ and $k$ are positive parameters, and the solution is real analytic. By employing a telescoping series argument, we are able to obtain the desired result. This novel approach provides fresh insights into the controllability of degenerate parabolic equations, offering alternative perspectives beyond the sole reliance on Carleman estimates.

As an important application, we consider a game problem for degenerate parabolic equations. Differential games were introduced originally by Isaacs (see
\cite{isaacs1965}). Since then,
lots of researchers were attracted to establish and improve the related theory. Meanwhile,
the theory was applied to a large number of fields. For a comprehensive survey on the
differential game theory, we refer to
\cite{bacsar1998dynamic,elliott1972existence,evans1984differential,friedman1971} and the references therein. The famous game problem, Stackelberg-Nash game problem, is an interesting question in the framework of mathematical finance. For instance, it is well
known that the price of an European call option is governed by a backward PDE. The independent space variable must be interpreted as the stock price and the time variable is in fact the reverse of time. In this regard, it can be interesting to control the solution of the system with the composed action
of several agents, each of them corresponding to a different range of values of the space variable. For further information on
the modeling and control of phenomena of this kind, see for instance \cite{ross2011elementary,wilmott1995mathematics}. For more Stackelberg–Nash game problems, we
refer the reader to the works in \cite{diaz2004approximate,guillen2013approximate}. In this paper, based on the Stackelberg–Nash null controllability, see \cite{araruna2015stackelberg,araruna2018stackelberg,araruna2017new}, we will focus on 
the characterizations of the set of Stackelberg–Nash equilibria and a norm optimal control problem. To the best of our knowledge, this paper is the first attempt to obtain the existence of the  Stackelberg–Nash equilibria by the Kakutani type fixed
point theorems for degenerate parabolic equations. Furthermore, we  characterize the sets of Stackelberg–Nash equilibria and consider a norm optimal control problem. After that, we can mention \cite{wang2018game}, which seems the first time to apply the Kakutani type fixed
point theorems to the game problems for heat equations. 


The remaining sections of this paper are organized as follows. In Section 2, we present the formulation of the main problems and state the key results: Theorem \ref{thm:obs-inq}, Theorem \ref{Intro-3}, Theorem \ref{Intro-4}, and Theorem \ref{9.15.1}. The proof of Theorem \ref{thm:obs-inq} is provided in Section 3. In Section 4, we delve into the discussion of a forthcoming Stackelberg-Nash game problem (to be formulated later) and provide the proofs for Theorem \ref{Intro-3} and Theorem \ref{Intro-4}. These findings are instrumental in the subsequent discussion of a norm optimal control problem (to be formulated later) and the proof of Theorem \ref{9.15.1} in Section 5.

\section{Problems formulation and main results}
In this section, we will outline the principal issues investigated and present the pivotal findings of this manuscript. Firstly, let us introduce necessary notations.

Let $T>0$ be a fixed positive time constant, and $I:=(0,1)$. Throughout this paper, we denote by $\langle\cdot,\cdot\rangle$ the scalar product in $L^2(I)$ and denote by $\|\cdot\|$ the norm induced by $\langle\cdot,\cdot\rangle$. We denote by $|\cdot|$ the Lebesgue measure on $\R$.

Let $A$ be an unbounded linear operator on $L^2(I)$:
\begin{equation}\label{operator}
\begin{cases}
\mathcal{D}(A):=\{v\in H_\alpha^1(I):(x^{\alpha}v_x)_x\in L^2(I) ~\text{and}~ BC_\alpha(v)=0 \},\\
Av :=  (x^{\alpha}v_x)_x,\ \forall v\in\mathcal{D}(A),\,\alpha\in(0,2),
\end{cases}
\end{equation}
where
$$
\begin{array}{ll}
H^1_\alpha(I):=\bigg\{v\in L^2(I):v \text{ is absolutely continuous in}~ I,
 x^{\frac{\alpha}{2}}v_x\in L^2(I)\,\text{and}~\,v(1)=0 \bigg\},
\end{array}
$$
and
\begin{equation*}
BC_\alpha(v)=
\begin{cases}
v_{|_{x=0}}, &\alpha\in(0,1),\\
(x^{\alpha}v_x)_{|_{x=0}}, &\alpha\in[1,2), 
\end{cases}
\end{equation*}
endowed with the norms
$$
\|v\|^2_{H^1_\alpha(I)}:=\|v\|^2+\|\sqrt{x^\alpha}v_x\|^2.
$$

The system, that we consider in this paper, is described by the following degenerate equation:
\begin{equation}
\label{eq:main}
\left\{
\begin{array}{ll}
y_t(x,t)-Ay(x,t) = 0, & \left( x ,t\right) \in I\times(0,T),   \\[2mm]
y(1,t)=BC_\alpha\big(y(\cdot,t)\big) = 0, & t\in \left(0,T\right), \\[2mm]
y\left(x, 0\right) =y_{0}(x), &  x\in I\,,
\end{array}
\right.
\end{equation}
where initial state $y_0 \in L^2(I)$, $y$ is the state variable. By \cite{cannarsa2008carleman}, one can check that for all $y_0\in L^2(I)$, system (\ref{eq:main}) admit a unique solution $y$ in the space $C([0,T];L^2(I))\cap L^2(0,T;H_\alpha^1(I))$.

The main result of this paper is the following observability inequality for the degenerate parabolic equation (\ref{eq:main}):
\begin{theorem}
  \label{thm:obs-inq}
Let $T>0$, $\alpha\in(0,2)$. Let $\mu\in(0,1)$ be defined as
\begin{equation}\label{u}
\mu=
\left\{
\begin{array}{lll}
\dfrac{3}{4}, &~\text{if}~ \alpha\in(0,2)\setminus\{1\},\\[3mm]
\dfrac{3}{2\gamma}  ~\text{for any}~\gamma\in(0,2),&~\text{if}~ \alpha=1,
\end{array}
\right.
\end{equation}
and $D\subset I\times(0,T)$ be a measurable subset with positive measures. Then there exists a constant $C=C(T,I, \alpha,D,\mu)\geq1$ such that the solution of equation (\ref{eq:main})  satisfies the following observability inequality: for any $y_0\in L^2(I)$,
\begin{equation}
  \label{eq:obs-inq}
  \|y(T)\|\le C \int_D|y(x,t)|dxdt.
\end{equation}
\end{theorem}
As an important application of Theorem \ref{thm:obs-inq}, we consider the following degenerate problem with three controls (one leader and two followers), but very similar considerations hold for other systems (degenerated cases or non-degenerated cases) with a
higher number of controls.
\begin{equation}
\label{Intro-1}
\left\{
\begin{array}{ll}
y_t(x,t) = Ay(x,t) +\chi_\omega g+\chi_{\omega_1} u_1+\chi_{\omega_2} u_2, & \left( x ,t\right) \in (0,1)\times(0,T),   \\[2mm]
y(1,t)=BC_\alpha\big(y(\cdot,t)\big) = 0, & t\in \left(0,T\right), \\[2mm]
y\left(x, 0\right) =y_{0}(x), &  x\in (0,1)\,,
\end{array}
\right.
\end{equation}
where $A$ defined in (\ref{operator}),  $\omega,\omega_1,\omega_2$ are measurable subsets of $(0,1)$ with positive measures. $\omega$ is the main (leader) control domain and $\omega_1,\omega_2$ are the secondary (followers) control domains. The controls $g,u_1,u_2\in L^\infty(0,T;L^2(0,1))$, where $g$ is the main (leader) control and $u_1$ and $u_2$ are the secondary (followers) controls.  Initial state is $y_0\in L^2(0,1)$. The solution of system (\ref{Intro-1}) denote by $y(\cdot;y_0,g,u_1,u_2)$.

We define the following admissible sets of controls:
\begin{equation*}
\cU_0:=\big\{g\in L^\infty(0,T;L^2(0,1)):\|g\|_{L^\infty(0,T;L^2(0,1))}\leq M_0  \big\},
\end{equation*}
and
\begin{equation*}
\cU_1:=\big\{u_1\in L^\infty(0,T;L^2(0,1)):\|u_1\|_{L^\infty(0,T;L^2(0,1))}\leq M_1  \big\},
\end{equation*}
\begin{equation*}
\cU_2:=\big\{u_2\in L^\infty(0,T;L^2(0,1)):\|u_2\|_{L^\infty(0,T;L^2(0,1))}\leq M_2  \big\},
\end{equation*}
where $M_0,M_1,M_2$ are positive constants. Meanwhile, we introduce the following two functionals: For each $i=1,2$,
the functional $J_i: (L^\infty(0,T;L^2(0,1)))^2\rightarrow [0,+\infty)$ is defined by
\begin{equation}\label{Intro-2}
J_i(u_1,u_2):= \|y(T;y_0,g,u_1,u_2)-y^i_T\|_{L^2(G_i)},
\end{equation}
where $G_i\subset(0,1)\ (i=1,2)$ be measurable subsets such that $\omega_i\subset G_i$ representing observation domains for the followers and  $y^i_T\in L^2(0,1)$ and
\begin{equation}\label{y1-y2}
y^1_T\not=y^2_T.
\end{equation}

The mathematical models for the control problems are:

1. The followers $u_1$ and $u_2$ assume that the leader $g $ has made a choice and intend to be a Nash equilibrium for the cost functionals $J_i(u_1,u_2),i=1,2$. Thus, once the leader $g$ has been fixed:

{\bf(P1)}\;\;Does there exist $(u_1^*,u_2^*)\in {\mathcal U}_1\times {\mathcal U}_2$ with respect to $g$ so that
\begin{equation}\label{20180208-1}
J_1(u_1^*,u_2^*)\leq J_1(u_1,u_2^*)\;\;\mbox{for all}\;\;u_1\in {\mathcal U}_1
\end{equation}
and
\begin{equation}\label{20180208-2}
J_2(u_1^*,u_2^*)\leq J_2(u_1^*,u_2)\;\;\mbox{for all}\;\;u_2\in {\mathcal U}_2?
\end{equation}

We call the problem {\bf(P1)} as a {\it Stackelberg-Nash game problem}. It is a noncooperative
game problem of the followers. If the answer to the problem {\bf(P1)} is yes, we call $(u_1^*,u_2^*)$ a {\it Stackelberg-Nash equilibrium} (or an {\it optimal
strategy pair}, or an {\it optimal control pair}) of {\bf(P1)}. We can understand the problem {\bf(P1)} in the following manner:
There are two followers executing their strategies and hoping to achieve their goals $y_T^1$ and $y_T^2$, respectively.
If the first follower chooses the strategy $u_1^*$, then the second follower can execute the strategy $u_2^*$
so that $y(T;y_0,g,u_1^*,u_2^*)$ is closer to $y_T^2$; Conversely, if the second follower chooses the
strategy $u_2^*$, then the first follower can execute the strategy $u_1^*$ so that $y(T;y_0,g,u_1^*,u_2^*)$ is closer
to $y_T^1$. Roughly speaking, if one follower is deviating from $(u_1^*,u_2^*)$, then the cost functional of this
follower would get larger; and there is no information given if both followers are deviating from
the Nash equilibrium $(u_1^*,u_2^*)$. From the knowledge of game theory, in generally,
Nash equilibrium is not unique.

2. Once the Stackelberg-Nash equilibrium $(u_1^*,u_2^*)$ has been identified for each leader $g$, denoted by $(u_1^*(g), u_2^*(g))$, we consider the following
norm optimal control problem:
\begin{equation}\label{9.11.2}
{\bf(P2)}\quad\ N(T,y_0):=\inf_{g\in \cU_0}\{\|g\|_{L^\infty(0,T;L^2(0,1))}:y(T;y_0,g,u_1^*(g),u_2^*(g))=0\}.
\end{equation}

The first main result of the
{\it Stackelberg-Nash game problem}  {\bf(P1)} is about the existence of a Stackelberg-Nash equilibrium.

\begin{theorem}
\label{Intro-3}
Let $g\in L^\infty(0,T;L^2(0,1))$ be given. The problem {\bf(P1)} with respect to the system (\ref{Intro-1}) admits a  Stackelberg-Nash equilibrium, i.e., there exist
at least a pair of $(u_1^*,u_2^*)\in {\mathcal{U}}_1\times {\mathcal{U}_2}$
with respect to $g$ so that (\ref{20180208-1}) and (\ref{20180208-2}) hold.
\end{theorem}

The second main result of the {\it Stackelberg-Nash game problem} {\bf(P1)} is the characterizations of the set of Stackelberg-Nash equilibria, which concerned with the following bang-bang property of the Stackelberg-Nash equilibria.
\begin{theorem}
\label{Intro-4}
Let $g\in L^\infty(0,T;L^2(0,1))$ be given and let $(u_1^*,u_2^*)\in {\mathcal{U}}_1\times {\mathcal{U}_2}$ be a Stackelberg-Nash equilibrium of the problem
{\bf(P1)} with respect to $g$. Then either
\begin{equation*}
\|u_1^*(t)\|_{L^2(G_1)}=M_1,~\text{for a.e.}~ t\in (0,T),
\end{equation*}
or
\begin{equation*}
\|u_2^*(t)\|_{L^2(G_2)}=M_2,~\text{for a.e.}~ t\in (0,T).
\end{equation*}
\end{theorem}
\begin{remark}
It is noteworthy that in previous studies addressing the Stackelberg-Nash game problem, the scholars have predominantly concentrated on the examination of system controllability, such as exact controllability \cite{araruna2015stackelberg}, approximate controllability \cite{diaz2004approximate}, null controllability \cite{araruna2018stackelberg}, and the presence of Stackelberg-Nash equilibrium, while paying no attention to the comprehensive characterization of the set of Stackelberg-Nash equilibrium and their distinctive bang-bang property.
\end{remark}

The result of the norm optimal control problem {\bf(P2)} is concered with the characterization of the leader. To the end, we make some assumptions:
\begin{itemize}
\item [$(A_1)$.] $\omega_i\subset\omega\subset(0,1)$ for $i=1,2$.
\item [$(A_2)$.] System \eqref{Intro-1} is null controllability, i.e. there exists a control $g\in L^\infty(0,T;L^2(0,1))$ such that
the corresponding solution $y$ of (\ref{Intro-1}) satisfies
$
y(T;y_0,g,u_1^*(g),u_2^*(g))=0,~\text{in}~ (0,1).
$
\end{itemize}
\begin{remark}
    The assumption $(A_1)$ is exclusively utilized for the purpose of characterizing the leader $g$, specifically in Theorem \ref{9.15.1}, which will be introduced later. The assumption $(A_2)$ is well-founded, supported by several  references, e.g., \cite{araruna2015stackelberg,araruna2018stackelberg,araruna2017new} demonstrating the null controllability of system \eqref{Intro-1}. Nevertheless, in the present section, our primary emphasis lies in delineating the properties of the leader $g$. Therefore, we have omitted the detailed proof of null controllability of system \eqref{Intro-1}, directing interested readers to consult the above references for further exploration.
\end{remark}

Let us introduce the following adjoint equation corresponding to (\ref{eq:main}):
\begin{equation}
\label{eq:adjoint}
\left\{
\begin{array}{ll}
z_t(x,t)+Az(x,t) = 0, & \left( x ,t\right) \in (0,1)\times(0,T),   \\[2mm]
z(1,t)=BC_\alpha\big(z(\cdot,t)\big) = 0, & t\in \left(0,T\right), \\[2mm]
z\left(x, T\right) =z_T(x), &  x\in (0,1)\,,
\end{array}
\right.
\end{equation}
where $z_T(x)\in L^2(0,1)$. 

Now, for every  $T\in\mathbb{R}^+$, let
	\begin{equation}\label{9.17.7}
	X_T=\{\chi_\omega z(\cdot; z_T)\bigm| z_T\in L^2(0,1)\},
	\end{equation}
	where $z(\cdot;z_T)$ is the solution of  system \eqref{eq:adjoint} with the terminal value  $z_T\in L^2(0,1)$.
	It is obviously that $X_T$ is a linear subspace of $L^1(0,T; L^2(0,1))$.
	Define
	\begin{equation}\label{9.18.1}
	Y_T=\overline{X_T}^{\|\cdot\|_{L^1(0,T; L^2(0,1))}}
	\end{equation}
	and denote
	\begin{equation}\label{9.18.2}
	Z_T=\left\{\chi_\omega z \in L^1(0,T; L^2(0,1))\ \left|\
	\begin{array}{lll}
	\mbox{for every } s\in (0,T), \mbox{ there exists } z_{T,s}\in L^2(0,1)\\
	\mbox{such that } z(\cdot)=e^{A(s-\cdot)}z_{T,s} \mbox{ on } [0,s]
	\end{array}\right.\right\}.
	\end{equation}

\begin{lemma}\label{lemma-g-exist}
Suppose $(A_2)$ holds. Let $T>0$, $y_0\in L^2(0,1)$.    Let Stackelberg-Nash equilibria $(u_1^*(g),u_2^*(g))$ be given with respect to $g$. Then there exists at least one $g^*\in L^\infty(0,T; L^2(0,1))$ such that
    $
N(T,y_0)=\|g^*\|_{L^\infty(0,T; L^2(0,1))}.
    $
\end{lemma}
The main result of the norm optimal control problem {\bf(P2)} is stated as follows.
\begin{theorem}\label{9.15.1}
Suppose $(A_1)$ and $(A_2)$ hold.		Let $y_0\in L^2(0,1)$, $T>0$. Let $(u_1^*(g^*),u_2^*(g^*))$ be given with respect to $g^*$ in Lemma \ref{lemma-g-exist} (i.e., $g^*$ is the optimal control for the problem {\bf (P2)}), and
		\begin{align*}
		J(\chi_\omega z)
		=&\frac{1}{2}\left[\int_0^T\left\| \chi_\omega z(t)\right\| dt\right]^2+\left\langle y_0, z(0)\right\rangle+\int_0^T\int_0^1\sum_{i=1}^2\chi_{\omega_i}u_i^*(g^*)z(t)dxdt\nonumber\\
		=&\frac{1}{2}\|\chi_\omega z\|_{L^1(0,T;L^2(0,1))}^2+\langle y_0,z(0)\rangle+\int_0^T\int_0^1\sum_{i=1}^2\chi_{\omega_i}u_i^*(g^*)z(t)dxdt,
\end{align*}
		where $\chi_\omega z\in Z_T$. Denote  a variational problem
		\begin{equation*}
		V(T,y_0)=\inf_{\chi_\omega z\in Z_T} J(\chi_\omega z).
		\end{equation*}
Then there exists a $\chi_\omega z_*\in Z_T$ such that
$$
V(T,y_0)=J(\chi_\omega z_*).
$$
Furthermore, we have the following conclusions:
 \begin{itemize}
     \item [ ($I$).] If $\chi_\omega z_*\neq0$, then
     $$
g^*(t)=\frac{\chi_\omega z_*(t)}{\|\chi_\omega z_*(t)\|}\|\chi_\omega z_*\|_{L^1(0,T; L^2(0,1))} ,\mbox{ for a.e. } t\in (0,T).
$$
     \item [ ($II$).] If $\chi_\omega z_*=0$, then
     $$
g^*(t)=0,\mbox{ for a.e. } t\in (0,T).
$$
 \end{itemize}
 Moreover,
		\begin{equation*}
		V(T,y_0)=-\frac{1}{2}N(T,y_0)^2.
		\end{equation*}
	\end{theorem}

\section{The proof of Theorem \ref{thm:obs-inq}}
The equation (\ref{eq:main}) will be referred to as the controlled system in our discussion, where the control in the form of $\chi_\omega u$ is applied on the right-hand side of the equation. Here, $u\in L^2(0,T;L^2(I))$ represents a function, and $\omega$ denotes a subset of $I$.

Let us recall an important result that will be used later, see Section 3.4 in \cite{buffe2018spectral}.
\begin{lemma}
\label{lemma-ob}
Let $\alpha\in(0,2)$ and $\omega_0\subset I$ be an open subset. Let $\mu\in(0,1)$ be defined in (\ref{u}). Then there exists a positive constant $C=C(\alpha,\omega_0,I,\mu)$ such that the solution of (\ref{eq:main})  satisfies the following observability inequality:
\begin{equation}
\label{open-ob}
\int_I|y(x,T)|^2dx\leq Ce^{CT^{-\frac{\mu}{1-\mu}}}\int_0^T\int_{\omega_0}|y(x,t)|^2dxdt,\,\forall\,y_0\in L^2(I).
\end{equation}
\end{lemma}
Next, we aim to quantify the real analyticity of the solution $y$ of equation \eqref{eq:main} on the subdomain that is away from the singularity point zero.
\begin{lemma}
\label{lemma-analytic}
Set $\alpha\in(0,2)$. Let $\omega$ be a subdomain of $I$ with $0\not\in \bar{\omega}$. Then there are positive constants $C=C(\alpha,I,\omega)\geq1$ and $\rho=\rho(\alpha,I,\omega)$, $0<\rho\leq1$, such that when $x\in\omega$, $\forall\,0\leq s< t$, $a\in\N$ and $\gamma\in\N$, the solution of equation (\ref{eq:main}) satisfies
\begin{equation}
\label{analytic-1}
|\partial_x^a\partial_t^\gamma y(x,t)|\leq \dfrac{Ce^{\frac{C}{t-s}}a!\gamma!}{\rho^a((t-s)/2)^\gamma}\|y(\cdot,s)\|.
\end{equation}
\end{lemma}
\begin{proof}
We borrow some ideas from \cite{liu2017observability}. We only need to prove the situation of $s=0$. Let $\{e_k\}_{k\geq1}$ and $\{w_k^2\}_{k\geq1}$ be respectively the sets of $L^2(I)$-normalized eigenfunctions and eigenvalues for $-\frac{\partial}{\partial x}\bigg(x^\alpha\frac{\partial}{\partial x}\bigg)$ with zero lateral Dirichlet boundary conditions; i.e.,
\begin{equation}
\label{lemma-analytic-eq1}
\left\{
\begin{array}{ll}
\big(x^\alpha (e_k(x))_x\big)_x+w^2_ke_k(x)=0, & \text{in}~\,\, I,   \\[2mm]
e_k =BC_\alpha(e_k) = 0, & \text{on}~\,\, \partial I,
\end{array}
\right.
\end{equation}
here, $0<w_1\leq w_2\leq\cdots\leq w_k\leq\cdots$ and $\lim_{k\rightarrow+\infty}w_k=+\infty$. Take $y_0=\sum_{k\geq1}a_ke_k$ with
\begin{equation}
\label{lemma-analytic-eq5}
\sum_{k\geq1}a_k^2<+\infty,
\end{equation}
and define
\begin{equation*}
y(x,\tau,t)=\sum_{k=1}^{+\infty} a_ke^{-w_k^2t+w_k\tau}e_k(x),\,\,\text{for}~ \,\,x\in I,\,\,t>0,\,\,\text{and}~\,\,\tau\in\R.
\end{equation*}
Then for all $t>0$ and $x\in I$,
$$
y(x,0,t)=\sum_{k=1}^{+\infty} a_ke^{-w_k^2t}e_k(x)=y(x,t)
$$
solves equation (\ref{eq:main}) with initial datum $y_0$ and
\begin{equation}
\label{lemma-analytic-eq2}
\partial_t^\gamma y(x,\tau,t)=\sum_{k\geq1} a_k(-w_k^2)^\gamma e^{-w_k^2t+w_k\tau}e_k(x),\,\,\text{for}~ \,\,x\in I,\,\,\gamma\in\N,\,\,\text{and}~\,\,\tau\in\R.
\end{equation}
Moreover, together with (\ref{lemma-analytic-eq1}), we have
\begin{equation}
\label{lemma-analytic-eq3}
\left\{
\begin{array}{ll}
\partial_\tau^2\big( \partial_t^\gamma y \big)+\frac{\partial}{\partial x}\bigg( x^\alpha\frac{\partial}{\partial x}\big( \partial_t^\gamma y \big) \bigg)=0, & \text{in}~\,\, I\times\R,   \\[2mm]
\partial_t^\gamma y=BC_\alpha(\partial_t^\gamma y)
 = 0, & \text{on}~\,\, \partial I\times\R.
\end{array}
\right.
\end{equation}
Because $0\not\in\bar{\omega}$, without loss of generality (otherwise use of a finite covering argument), we can assume that there exists
$R\leq\frac{1}{8}$, $x_0\in\omega$ such that
\begin{equation}\label{lemma-analytic-eq9}
\omega\subset D_R(x_0),\,D_{2R}(x_0)\subset I, ~\text{and}~D_{2R}(x_0)\cap D_R(0)=\emptyset,
\end{equation}
where $$D_R(x_0)=\{x\in\omega:|x-x_0|\leq R  \}.$$
By the last equality of (\ref{lemma-analytic-eq9}), it shows the function $x^\alpha$ is real analytic and non-degenerate in $D_{2R}(x_0,0)\subset I\times\R$. By the real analytic estimates of solutions to linear elliptic equations with real analytic coefficients (cf. for instance, \cite{apraiz2013null}, Chapter 3 in \cite{john2004plane}, and Chapter 5 in \cite{morrey2009multiple}), we have that there are constants $C=C(R,\alpha)\geq1$ and $\rho=\rho(R,\alpha)\in(0,1)$ such that any solutions to equation (\ref{lemma-analytic-eq3}) satisfies,
\begin{equation}
\label{lemma-analytic-eq4}
\|\partial_x^a\partial_t^\gamma y(\cdot,\cdot,t)\|_{L^\infty(D_R(x_0,0))}\leq\frac{Ca!}{\rho^a}\bigg( \int_{D_{2R}(x_0,0)} |\partial_t^\gamma y(x,\tau,t)|^2dxd\tau \bigg)^{\frac{1}{2}},
\end{equation}
where $a\in \N$, $\gamma\in\N$. Notice that for each $t > 0$,
\begin{equation}
\begin{split}
\label{lemma-analytic-eq10}
\int_{D_{2R}(x_0,0)} |\partial_t^\gamma y(x,\tau,t)|^2dxd\tau&\leq\int_{-2R}^{2R}\int_{D_{2R}(x_0)} |\partial_t^\gamma y(x,\tau,t)|^2dxd\tau\\
&\leq\int_{-2R}^{2R}\int_I |\partial_t^\gamma y(x,\tau,t)|^2dxd\tau.
\end{split}
\end{equation}
By the orthogonality of $\{e_k\}_{k\geq1}$ in $L^2(I)$, and \eqref{lemma-analytic-eq2}, we have
\begin{equation*}
\begin{split}
\int_I|\partial_t^\gamma y(x,\tau,t)|^2dx&=\int_I\bigg|\sum_{k\geq1} a_k(-w_k^2)^\gamma e^{-w_k^2t+w_k\tau}e_k(x)\bigg|^2dx\\
&=\sum_{k\geq1} a_k^2w_k^{4\gamma}e^{-2w_k^2t+2w_k\tau}\\
&=\sum_{k\geq1} a_k^2w_k^{4\gamma}e^{-w_k^2t}e^{-w_k^2t+2w_k\tau}\\
&\leq \max_{k\geq1}\{ w_k^{4\gamma}e^{-w_k^2t} \}\max_{k\geq1}\{ e^{-w_k^2t+2w_k\tau} \}\sum_{k\geq1} a_k^2,
\end{split}
\end{equation*}
which, along with \eqref{lemma-analytic-eq10}, it implies
\begin{equation}
\begin{split}
\label{lemma-analytic-eq11}
\int_{D_{2R}(x_0,0)} |\partial_t^\gamma y(x,\tau,t)|^2dxd\tau&\leq\max_{k\geq1}\{ w_k^{4\gamma}e^{-w_k^2t} \}\max_{k\geq1}\{ e^{-w_k^2t+4Rw_k} \}\sum_{k\geq1} a_k^2.
\end{split}
\end{equation}
On one hand, from Stirling's formula, see page 111 in \cite{krantz2002primer}, i.e. for a absolute constant $C$,
$\gamma^{\gamma}\leq Ce^{\gamma}(\gamma)!$, we have
\begin{equation*}
\max_{k\geq1}\{ w_k^{4\gamma}e^{-w_k^2t} \}\leq C\bigg(\frac{2}{t}\bigg)^{2\gamma}(\gamma!)^2.
\end{equation*}
On the other hand, we have
\begin{equation*}
\max_{k\geq1}\{ e^{-w_k^2t+4Rw_k} \}=e^{\frac{4R^2}{t}}.
\end{equation*}
Therefore, it stands
\begin{equation}
\label{lemma-analytic-eq6}
\int_{D_{2R}(x_0,0)} |\partial_t^\gamma y(x,\tau,t)|^2dxd\tau\leq C\bigg(\frac{2}{t}\bigg)^{2\gamma}(\gamma!)^2 e^{\frac{4R^2}{t}}\sum_{k\geq1} a_k^2.
\end{equation}
This together with (\ref{lemma-analytic-eq4}), it implies
$$
\|\partial_x^a\partial_t^\gamma y(\cdot,t)\|_{L^\infty(D_R(x_0,0))}\leq\frac{Ce^{\frac{2R^2}{t}}a!\gamma!}{\rho^a(t/2)^\gamma}\|y_0\|,
$$
which, along with the first result of (\ref{lemma-analytic-eq9}), we have
$$
\|\partial_x^a\partial_t^\gamma y(\cdot,t)\|_{L^\infty(\omega)}\leq\frac{Ce^{\frac{2R^2}{t}}a!\gamma!}{\rho^a(t/2)^\gamma}\|y_0\|.
$$
Hence, the above inequality implies that (\ref{analytic-1}) with $s=0$ holds. The proof is completed.
\end{proof}

The subsequent two lemmas focus on the propagation of smallness estimates from measurable sets for real analytic functions. This type of observability estimate was initially established in \cite{vessella1999continuous} (also refer to \cite{apraiz2013null} and \cite{apraiz2014observability} for associated findings). To facilitate subsequent references, one lemma pertains to the one-dimensional scenario, while the other lemma caters to the multidimensional case. For a more streamlined proof, interested readers can refer to Section 3 of \cite{apraiz2014observability}.
\begin{lemma}
\label{lemma-measurable-1}
Let $g:[a,a+s]\rightarrow\R$, where $a\in\R$ and $s>0$ be a real analytic function satisfying
\begin{equation*}
\left| \frac{d^k}{dx^k}g(x) \right|\leq Mk!(s\rho)^{-k},\forall\,x\in[a,a+s],\,\,\forall\,k\in\N,
\end{equation*}
where $M>0$, $\rho\in(0,1]$. Assume that $F\subset[a,a+s]$ is a measurable subset of positive measure. Then there are $C=C(\rho, |F|/s)$ and $\theta=\theta(\rho, |F|/s)$, $0 <\theta<1$, such that
\begin{equation*}
\|g\|_{L^\infty(a,a+s)}\leq CM^{1-\theta} \bigg(\frac{1}{|F|} \int_F|g(x)|dx \bigg)^{\theta}.
\end{equation*}
\end{lemma}
\begin{lemma}
\label{lemma-measurable}
Let $\omega$ be a bounded domain in $\R^n,n\geq1$ and $\bar{\omega}\subset\omega$ be a measurable set of positive measure. Let $f$ be an analytic function in $\omega$ satisfying
\begin{equation*}
|\partial_x^af(x)|\leq M \rho^{-|a|}|a|!,\,\,\forall\,x\in\omega,\,\,\forall\,a\in\N^n,
\end{equation*}
where $M>0$, $\rho\in(0,1]$. Then there are $C=C(|\omega|, \rho, |\bar{\omega}|)$ and $\theta=\theta(|\omega|, \rho, |\bar{\omega}|)$, $0 <\theta<1$, such that
\begin{equation*}
\|f\|_{L^\infty(\omega)}\leq CM^{1-\theta} \bigg( \int_{\bar{\omega}}|f(x)|dx \bigg)^{\theta}.
\end{equation*}
\end{lemma}
Next, we shall make use of Lemma \ref{lemma-ob} and Lemma \ref{lemma-analytic}, as well as Lemma \ref{lemma-measurable-1} and Lemma \ref{lemma-measurable}, to establish an interpolation inequality from
measurable sets for any solution $y$ to equation \eqref{eq:main}. For similar results, we refer the reader to \cite{apraiz2014observability} or \cite{liu2017observability}.
\begin{theorem}
\label{thm:inter}
Let $0\leq t_1<t_2<1$, $\eta\in(0,1)$, $\mu\in(0,1)$ defined in \eqref{u}, $\alpha\in(0,2)$ and $\sigma>0$. Assume that $E\subset(t_1,t_2)$ is a measurable subset with $|E\cap(t_1,t_2)|\geq\eta(t_2-t_1)$ such that for all $t\in E$, the measurable subset $D_t\subset\omega$ stands that $|D_t|\geq\sigma$. Then, there are constants $C=C(I,\omega,\alpha,\eta,\sigma,\mu)\geq1$ and $\theta=\theta(I,\omega,\alpha,\eta,\sigma,\mu)\in(0,1)$ such that the solution to equation (\ref{eq:main}) satisfies
\begin{equation}\label{thm-1}
\|y(\cdot,t_2)\|\leq \left( \int_{t_1}^{t_2}\chi_E \|y(\cdot,t)\|_{L^1(D_t)}dt\right)^\theta\left( e^{C(t_2-t_1)^{-\frac{\mu}{1-\mu}}}\|y(\cdot,t_1)\| \right)^{1-\theta}.
\end{equation}
\end{theorem}

\begin{proof}
Define
$$
\tau=t_1+\frac{\eta}{3}(t_2-t_1),\,\,F=E\cap(\tau,t_2).
$$
One can verify that$|F|>\frac{\eta}{2}(t_2-t_1)$. This shows that $F\subset[\tau,t_2]$ is a measurable subset of positive measure.

From Lemma \ref{lemma-analytic}, we have that there are constants $C=C(I,\omega,\alpha,\eta,\sigma)\geq1$ and $\theta=\theta(I,\omega,\alpha,\eta,\sigma)\in(0,1)$ such that for all $t\in[\tau,t_2]$ and for all $x\in\omega$, $a\in\N$, $\gamma\in\N$,
\begin{equation}\label{thm-2}
|\partial_x^a\partial_t^\gamma y(\cdot,t)|\leq \dfrac{e^{\frac{C}{t_2-t_1}}a!\gamma!}{\rho^a(\eta(t_2-t_1)/2)^\gamma}\|y(\cdot,t_1)\|.
\end{equation}
Now, for simplicity, we write
\begin{equation}\label{thm-3}
M=e^{\frac{C}{t_2-t_1}}\|y(\cdot,t_1)\|.
\end{equation}
Then for all $x\in\omega$ and $\gamma\in\N$, by (\ref{thm-2}) with $a=0$, we have
\begin{equation}\label{thm-4}
|\partial_t^\gamma y(\cdot,t)|\leq M\gamma!\left(\frac{\eta(t_2-t_1)}{2}\right)^{-\gamma}\leq M\gamma!\left(\frac{\eta(t_2-\tau)}{6}\right)^{-\gamma},~\text{for all}~t\in[\tau,t_2].
\end{equation}
Hence, it follows from Lemma \ref{lemma-measurable-1} that there are constants $C=C(\eta)\geq1$ and $\theta=\theta(\eta)\in(0,1)$ such that for all $x\in\omega$,
\begin{equation}\label{thm-5}
\|y(x,\cdot)\|_{L^\infty(\tau,t_2)}\leq CM^{1-\theta}\left( \frac{1}{|F|}\int_F|y(x,t)|dt \right)^\theta.
\end{equation}
On the other hand, we obtain from Lemma \ref{lemma-ob} that there exists a constant $C=C(I,\omega,\alpha,\mu)$ such that the following observability inequality holds:
\begin{equation}\label{thm-6}
\int_I|y(x,t_2)|^2dx\leq e^{C(t_2-\tau)^{-\frac{\mu}{1-\mu}}}\int_\tau^{t_2}\int_{\omega}|y(x,t)|^2dxdt.
\end{equation}
By (\ref{thm-2}) with $a=\gamma=0$, we have
$$
\|y\|_{L^2(\tau,t_2;L^2(\omega))}\leq\|y\|^{\frac{1}{2}}_{L^\infty(\tau,t_2;L^\infty(\omega))}\|y\|^{\frac{1}{2}}_{L^1(\tau,t_2;L^1(\omega))}\leq
e^{\frac{C}{t_2-t_1}}\|y(\cdot,t_1)\|^{\frac{1}{2}}\|y\|^{\frac{1}{2}}_{L^1(\tau,t_2;L^1(\omega))},
$$
which along with (\ref{thm-6}), it holds
\begin{equation}\label{thm-7}
\begin{split}
\|y(\cdot,t_2)\|&\leq\|y\|^{\frac{1}{2}}_{L^1(\tau,t_2;L^1(\omega))} e^{C(t_2-\tau)^{-\frac{\mu}{1-\mu}}}e^{\frac{C}{t_2-t_1}}\|y(\cdot,t_1)\|^{\frac{1}{2}}\\
&\leq \|y\|^{\frac{1}{2}}_{L^1(\tau,t_2;L^1(\omega))}\left( e^{C(t_2-t_1)^{-\frac{\mu}{1-\mu}}}\|y(\cdot,t_1)\| \right)^{\frac{1}{2}}.
\end{split}
\end{equation}
It follows by (\ref{thm-5}) that for any $t\in F$,
\begin{equation}\label{thm-8}
\begin{split}
\|y\|_{L^1(\tau,t_2;L^1(\omega))}&\leq(t_2-\tau)\int_\omega\|y(x,\cdot)\|_{L^\infty(\tau,t_2)}dx
\leq CM^{1-\theta}\int_\omega\left(\int_F|y(x,t)|dt \right)^\theta dx\\
&\leq CM^{1-\theta}\left(\int_\omega\int_F|y(x,t)|dt dx\right)^\theta.
\end{split}
\end{equation}
Finally, by (\ref{thm-2}) with $\gamma=0$, we infer that for any $t \in F$,
$$
|\partial_x^ay(\cdot,t)|\leq Ma!\rho^{-a},~\text{for all}~x\in\omega.
$$
Since the measurable subset $D_t\subset\omega$ stands that $|D_t|\geq\sigma$, when $t\in F$, we obtain from Lemma \ref{lemma-measurable} (with $n=1$) that there are constants $C_1=C_1(\omega,\rho,\sigma)\geq1$ and $\theta_1=\theta_1(\omega,\rho,\sigma)\in(0,1)$ such that
$$
\|y(\cdot,t)\|_{L^\infty(\omega)}\leq C_1M^{1-\theta_1}\left( \int_{D_t}|y(x,t)|dx \right)^{\theta_1},
$$
which along with (\ref{thm-8}), it implies
\begin{equation}\label{thm-9}
\begin{split}
\|y\|_{L^1(\tau,t_2;L^1(\omega))}&\leq CM^{1-\theta}\left[\int_FC_1M^{1-\theta_1}\left(\int_{D_t}|y(x,t)| dx\right)^{\theta_1} dt\right]^\theta\\
&\leq CC_1M^{1-\theta\theta_1}\left( \int_F\int_{D_t}|y(x,t)| dxdt \right)^{\theta\theta_1}.
\end{split}
\end{equation}
Therefore, by (\ref{thm-7}), it implies
$$
\|y(\cdot,t_2)\|\leq CC_1M^{1-\theta\theta_1}\left( \int_F\int_{D_t}|y(x,t)| dxdt \right)^{\theta\theta_1}\left( e^{C(t_2-t_1)^{-\frac{\mu}{1-\mu}}}\|y(\cdot,t_1)\| \right)^{\frac{1}{2}}.
$$
Set $\theta'=\theta\theta_1\in(0,1)$ and by (\ref{thm-3}), we obtain the desired estimate (\ref{thm-1}).
\end{proof}
By leveraging Fubini's theorem and the property of Lebesgue density points for measurable subsets, along with the interpolation inequality established in Theorem \ref{thm:inter}, a telescoping series can be constructed with respect to any solution $y$ of \eqref{eq:main}. Subsequently, the desired observability inequality from measurable subsets can be derived directly from this telescoping series (also see \cite{apraiz2014observability} and \cite{phung2013observability}). To begin, we quote the following fact from \cite{phung2013observability}, which pertains to the property of Lebesgue density points and represents an improved version of Lemma 2.1.5 in \cite{fattorini2005infinite}.
\begin{lemma}\label{lemma-1}
Let $E\subset(0,T)$ be a measurable subset of positive measure. Assume that $l\in(0,T)$ be a Lebesgue density point of $E$. Then, for each
$q\in(0,1)$, there exists a sequence $\{\ell_n\}_{n\geq1}\subset[0,T]$, which monotone decreasing converges to $l$, such that for any $n\geq1$,
\begin{equation}\label{lemma-1-1}
\ell_{n+1}-\ell_{n+2}=q(\ell_{n}-\ell_{n+1}),
\end{equation}
and
\begin{equation}\label{lemma-1-2}
|E\cap (\ell_{n+1},\ell_n)|\geq \frac{\ell_n-\ell_{n+1}}{3}.
\end{equation}
\end{lemma}
Now, we are prepared to present the proof of Theorem \ref{thm:obs-inq} by employing the techniques outlined in Theorem \ref{thm:inter} and Lemma \ref{lemma-1}, in addition to the telescoping series method.

\begin{proof}[\textbf{Proof of Theorem~\ref{thm:obs-inq}}]
Firstly, we can make the assumption, without loss of generality, that $T \in (0, 1]$ and that there exists a subdomain $\omega \subset I$ such that $0 \notin \bar{\omega}$. (Alternatively, we can select a new measurable subset $\tilde{D} \subset \left(I \setminus \{0\}\right) \times (T_1, T_2)$ with $0 \leq T_1 < T_2 \leq T$ and $T_2 - T_1 \leq 1$.) Next, for a.e. $t\in(0,T)$, we slice the space
\begin{equation}\label{thm:obs-inq-1}
D_t=\{ x\in\omega:(x,t)\in D \}.
\end{equation}
Define
\begin{equation}\label{thm:obs-inq-2}
E=\left\{ t\in(0,T):|D_t|\geq\frac{|D|}{2T} \right\}.
\end{equation}
From the Fubini theorem, we have
$$
|D|=\int_0^T|D_t|dt=\int_E|D_t|dt+\int_{(0,T)\setminus E}|D_t|dt\leq |E||\omega|+\frac{|D|}{2},
$$
which shows $|E|\geq \frac{|D|}{2|\omega|}$, i.e. $E$ is a measurable subset with positive measures. Moreover, we have
\begin{equation}\label{thm:obs-inq-3}
\chi_E(t)\chi_{D_t}(x)\leq\chi_D(x,t),\,\,a.e.\,(x,t)\in I\times(0,T).
\end{equation}
Let $l\in(0,T)$ be a Lebesgue density point of $E$. Then for each $q \in(0, 1)$, which is to be fixed later, by Lemma \ref{lemma-1}, there exists a monotone decreasing sequence $\{\ell_n\}_{n\geq1}$ such that (\ref{lemma-1-1}) and (\ref{lemma-1-2}) hold. Moreover,
\begin{equation}\label{thm:obs-inq-4}
\lim_{n\rightarrow+\infty}l_n=l.
\end{equation}
We can obtain from Theorem \ref{thm:inter} that there are positive constants $C=C(I,\omega,\alpha,\eta,\sigma,\mu)\geq1$ and $\theta=\theta(I,\omega,\alpha,\eta,\sigma,\mu)\in(0,1)$ such that for all $n\geq1$,
\begin{equation}\label{thm:obs-inq-5}
\|y(\cdot,l_n)\|\leq \left( \int_{l_{n+1}}^{l_n}\chi_E \|y(\cdot,t)\|_{L^1(D_t)}dt\right)^\theta\left( e^{C(l_n-l_{n+1})^{-\frac{\mu}{1-\mu}}}\|y(\cdot,l_{n+1})\| \right)^{1-\theta}.
\end{equation}
Using the Young inequality with $\epsilon$, i.e., $ab\leq\epsilon a^{(1-\theta)^{-1}}+\epsilon^{-\frac{1-\theta}{\theta}}b^{\theta^{-1}}$, the above inequality leads to
\begin{equation}\label{thm:obs-inq-6}
\|y(\cdot,l_n)\|\leq \epsilon\|y(\cdot,l_{n+1})\|+\epsilon^{-\frac{1-\theta}{\theta}}e^{C(l_n-l_{n+1})^{-\frac{\mu}{1-\mu}}}\int_{l_{n+1}}^{l_n}\chi_E \|y(\cdot,t)\|_{L^1(D_t)}dt.
\end{equation}
Multiplying the above inequality by $\epsilon^{\frac{1-\theta}{\theta}} e^{-C(\ell_{n}-\ell_{n+1})^{-\frac{\mu}{1-\mu}}}$, we have
\begin{align*}
 &\epsilon^{\frac{1-\theta}{\theta}}e^{-C(\ell_{n}-\ell_{n+1})^{-\frac{\mu}{1-\mu}}}\|y(\cdot,l_n)\|-\epsilon^{\frac{1}{\theta}}e^{-C(\ell_{n}-\ell_{n+1})^{-\frac{\mu}{1-\mu}}}
\|y(\cdot,l_{n+1})\|\\
\leq& \int_{l_{n+1}}^{l_n}\chi_E \|y(\cdot,t)\|_{L^1(D_t)}dt.   
\end{align*}
Choosing $\epsilon =e^{-\theta(\ell_{n}-\ell_{n+1})^{-\frac{\mu}{1-\mu}}}$ in the above inequality, we get
\begin{align*}
&e^{-(C+1-\theta)(\ell_{n}-\ell_{n+1})^{-\frac{\mu}{1-\mu}}}\|y(\cdot,l_n)\|-e^{-(C+1)(\ell_{n}-\ell_{n+1})^{-\frac{\mu}{1-\mu}}}
\|y(\cdot,l_{n+1})\|\\
\leq& \int_{l_{n+1}}^{l_n}\chi_E \|y(\cdot,t)\|_{L^1(D_t)}dt.
\end{align*}

Now, fixing $q = \bigg(\frac{C+1-\theta}{C+1}\bigg)^{\frac{1-\mu}{\mu}}\in(0,1)$, it follows from (\ref{lemma-1-1}) and
the latter inequality that
\begin{align*}
&e^{-(C+1-\theta)(\ell_{n}-\ell_{n+1})^{-\frac{\mu}{1-\mu}}}\|y(\cdot,l_n)\|-e^{-(C+1-\theta)(\ell_{n+1}-\ell_{n+2})^{-\frac{\mu}{1-\mu}}}
\|y(\cdot,l_{n+1})\|\\
\leq& \int_{l_{n+1}}^{l_n}\chi_E \|y(\cdot,t)\|_{L^1(D_t)}dt.
\end{align*}
To summarize the inequalities earlier from $n=1$ to $+\infty$ and noting that (\ref{thm:obs-inq-4}), we conclude that
$$
\|y(\cdot,l_1)\|\leq e^{(C+1-\theta)(\ell_{1}-\ell_{2})^{-\frac{\mu}{1-\mu}}}\int_{l}^{l_1}\chi_E(t) \|y(\cdot,t)\|_{L^1(D_t)}dt,
$$
which along with (\ref{thm:obs-inq-3}) and $\|y(\cdot,T)\|\leq C(T)\|y(\cdot,l_1)\|$, we obtain the desired estimate (\ref{eq:obs-inq}). The proof is completed.

\end{proof}

\section{A  Stackelberg-Nash game problem}
In this section, we shall discuss the Stackelberg-Nash game problem {\bf (P1)}.

\subsection{Existence of Stackelberg-Nash equilibrium}

In this subsection, we will prove Theorem~\ref{Intro-3}. Its proof needs the Kakutani fixed point theorem
quoted from \cite{charalambos2013infinite}.

\begin{lemma}\label{Intro-5}  Let $S$ be a nonempty, compact and convex subset of
a locally convex Hausdorff space $X$. Let $\Phi: S\mapsto 2^S$ (where $2^S$ denotes the set consisting of
all subsets of $S$)
be a set-valued function
satisfying:
\begin{itemize}
\item[(i)] For each $s\in S$, $\Phi(s)$ is a nonempty and convex subset;
\item[(ii)] Graph $\Phi:=\{(s,z): s\in S\;\mbox{and}\;z\in \Phi(s)\}$ is closed.
\end{itemize}
Then the set of fixed points of $\Phi$ is nonempty and compact, where $s^*\in S$ is called to
be a \emph{fixed point} of $\Phi$ if $s^*\in \Phi(s^*)$.
\end{lemma}
Now, we are in a position to prove Theorem ~\ref{Intro-3}.

\begin{proof}[\textbf{Proof of Theorem~\ref{Intro-3}}]
We first introduce  three set-valued functions
$\Phi_1: {\mathcal{U}_1}\mapsto 2^{{\mathcal{U}_2}}, \Phi_2: {\mathcal{U}_2}\mapsto 2^{{\mathcal{U}_1}}$
and $\Phi: {\mathcal{U}_1}\times {\mathcal{U}_2}\mapsto 2^{{\mathcal{U}_1}\times {\mathcal{U}_2}}$
as follows:
\begin{equation}\label{result-1}
\Phi_1 u_1:= \{u_2\in {\mathcal U}_2: J_2(u_1,u_2)\leq J_2(u_1,v_2)\;\;
\mbox{for all}\;\;v_2\in {\mathcal U}_2\},\;\;u_1\in {\mathcal U}_1,
\end{equation}
\begin{equation}\label{result-2}
\Phi_2 u_2:= \{u_1\in {\mathcal U}_1: J_1(u_1,u_2)\leq J_1(v_1,u_2)\;\;
\mbox{for all}\;\;v_1\in {\mathcal U}_1\},\;\;u_2\in {\mathcal U}_2,
\end{equation}
and
\begin{equation}\label{result-3}
\Phi(u_1,u_2):= \{({\widetilde u}_1,{\widetilde u}_2): {\widetilde u}_1\in \Phi_2 u_2\;\;\mbox{and}\;\;
{\widetilde u}_2\in \Phi_1 u_1\},\;\;(u_1,u_2)\in {\mathcal U}_1\times {\mathcal U}_2.
\end{equation}
Then we set
\begin{equation*}
X:= (L^\infty(0,T;L^2(0,1)))^2\;\;\mbox{and}\;\;S:= {\mathcal U}_1\times {\mathcal U}_2.
\end{equation*}
It is clear that $X$ is a locally convex Hausdorff space.
The rest of the proof will be carried out by the following four steps.

{\it Step 1. Show that $S$ is a nonempty, compact and convex subset of $X$.}

This fact can be easily checked. We omit the proof here.

{\it Step 2. $\Phi(u_1,u_2)$ is nonempty.} We prove that $\Phi(u_1,u_2)$ is nonempty for each $(u_1,u_2)\in S$.

We arbitrarily fix $(u_1,u_2)\in S$. According to (\ref{result-1})-(\ref{result-3}), it suffices to
show that $\Phi_1 u_1$ and $\Phi_2 u_2$ are nonempty. For this purpose, we introduce the following
auxiliary optimal control problem:
\begin{equation*}
{\rm\bf (P_{au})}\;\;\;\;\;\;\;\inf_{v_2\in {\mathcal U}_2} J_2(u_1,v_2).
\end{equation*}
Let
\begin{equation}\label{result-4}
d:=\inf_{v_2\in {\mathcal U}_2} J_2(u_1,v_2).
\end{equation}
It is obvious that $d\geq 0$. Let $\{v_{2,n}\}_{n\geq 1}\subseteq {\mathcal U}_2$ be a minimizing sequence
so that
\begin{equation}\label{result-5}
d=\lim_{n\rightarrow \infty} J_2(u_1,v_{2,n}).
\end{equation}
On one hand, since $\|v_{2,n}\|_{L^\infty(0,T;L^2(0,1))}\leq M_2$,
there exists a subsequence of $\{v_{2,n}\}_{n\geq 1}$, still denoted by itself, and
$v_{2,0}\in {\mathcal U}_2$, so that
\begin{equation}\label{result-5:1}
v_{2,n}\rightarrow v_{2,0}\;\;\mbox{weakly star in}\;\;L^\infty(0,T;L^2(0,1)).
\end{equation}

We denote by $\bar{y}_n(\cdot):=y(\cdot; y_0,g,u_1,v_{2,n})-y(\cdot; y_0,g,u_1,v_{2,0})$. According to (\ref{Intro-1}), it is clear that
\begin{equation}
\label{result-5:2}
\left\{
\begin{array}{ll}
\bar{y}_{n,t}(x,t) = A\bar{y}_n(x,t) +\chi_{\omega_2}( v_{2,n}-v_{2,0}), & \left( x ,t\right) \in (0,1)\times(0,T),   \\[2mm]
\bar{y}_n(1,t)=BC_\alpha\big(\bar{y}_n(\cdot,t)\big) = 0, & t\in \left(0,T\right), \\[2mm]
\bar{y}_n\left(x, 0\right) =0, &  x\in (0,1).
\end{array}
\right.
\end{equation}
Multiplying the first equation of (\ref{result-5:2}) by $\bar{y}_n$ and integrating over $(0, 1)$, one deduces, for all $t \leq T$,
\begin{align}
\label{result-5:3}
\frac{1}{2}\frac{d}{dt}\|\bar{y}_n(t)\|^2+\|x^{\frac{\alpha}{2}}\bar{y}_{n,x}(t)\|^2&\leq \frac{1}{2}\|\bar{y}_n(t)\|^2+\frac{1}{2}\|(v_{2,n}-v_{2,0})(t)\|^2_{L^2(\omega_2)}.
\end{align}
By Gronwall's inequality, we have
\begin{align}
\label{result-5:4}
\sup_{t\in[0,T]}\|\bar{y}_n(t)\|^2\leq e^{T}\|v_{2,n}-v_{2,0}\|^2_{L^2(0,T;L^2(\omega_2))}.
\end{align}
Notice that, from (\ref{result-5:3}) and (\ref{result-5:4}), it also follows that
\begin{align}
\label{result-5:5}
\|x^{\frac{\alpha}{2}}\bar{y}_{n,x}\|^2_{L^2(0,T;L^2(0,1))}\leq \left[\frac{ e^{T}T}{2}+\frac{1}{2}\right]\|v_{2,n}-v_{2,0}\|^2_{L^2(0,T;L^2(\omega_2))}.
\end{align}
From these, we obtain
\begin{align}
\label{result-5:6}
\sup_{t\in[0,T]}\|\bar{y}_n(t)\|^2+\|x^{\frac{\alpha}{2}}\bar{y}_{n,x}\|^2_{L^2(0,T;L^2(0,1))}\leq C\|v_{2,n}-v_{2,0}\|^2_{L^2(0,T;L^2(\omega_2))},
\end{align}
where $C=C(T)$ independent of $n$.

It follows from (\ref{result-5:1}) and (\ref{result-5:6}) that there exists a subsequence of $\bar{y}_n$, still denoted by itself, and $\bar{y}\in C([0,T];L^2(0,1))\cap L^2(0,T;H_\alpha^1(0,1))$ such that
\begin{align}
\label{result-5:7}
\bar{y}_n\rightarrow \bar{y}\;\;\mbox{weakly in}\;\;C([0,T];L^2(0,1))\cap L^2(0,T;H_a^1(0,1)).
\end{align}
Since the embedding $H_\alpha^1(0,1)\hookrightarrow L^2(0,1)$ is compact, see \cite{sun2022fundamental}, also see section 6 in \cite{alabau2006carleman}, it implies that
\begin{align}
\label{result-5:73}
\bar{y}_n\rightarrow \bar{y}\;\;\mbox{strongly in}\;\;L^2(0,1).
\end{align}
Passing to the limit for $n\rightarrow\infty$ in (\ref{result-5:2}), by (\ref{result-5:1}) and (\ref{result-5:7}), we have
\begin{equation*}
    \left\{
\begin{array}{ll}
\bar{y}_{t}(x,t) = A\bar{y}(x,t), & \left( x ,t\right) \in (0,1)\times(0,T),   \\[2mm]
\bar{y}(1,t)=BC_\alpha\big(\bar{y}(\cdot,t)\big) = 0, & t\in \left(0,T\right), \\[2mm]
\bar{y}\left(x, 0\right) =0, &  x\in (0,1),
\end{array}
\right.
\end{equation*}
from which we obtain that $\bar{y} = 0$. Hence,
on the other hand, we have
\begin{equation}\label{result-6}
y(T; y_0,g,u_1,v_{2,n})\rightarrow y(T; y_0,g,u_1,v_{2,0}) ~\text{strongly in}~ L^2(0,1).
\end{equation}
It follows from (\ref{result-5}), (\ref{Intro-2}),  (\ref{result-6}) that
\begin{equation}\label{result-7}
d=J_2(u_1,v_{2,0}).
\end{equation}
Noting that $v_{2,0}\in {\mathcal U}_2$, by (\ref{result-4}), (\ref{result-7}) and (\ref{result-1}),
we obtain that $v_{2,0}\in \Phi_1 u_1$. This implies that $\Phi_1 u_1\not=\emptyset$. In the same way,
we also have that $\Phi_2 u_2\not=\emptyset$.

{\it Step 3. Convex subset $\Phi(u_1,u_2)$.} We show that $\Phi(u_1,u_2)$ is a convex subset of ${\mathcal U}_1\times {\mathcal U}_2$ for
each $(u_1,u_2)\in {\mathcal U}_1\times {\mathcal U}_2$.

We arbitrarily fix $(u_1,u_2)\in {\mathcal U}_1\times {\mathcal U}_2$.
According to (\ref{result-1})-(\ref{result-3}), it suffices to prove that
$\Phi_1 u_1$ is a  convex subset of ${\mathcal U}_2$. The convexity of $\Phi_2 u_2$ can
be similarly proved. For this purpose, we arbitrarily fix ${\widetilde u}_2, {\widehat u}_2\in \Phi_1 u_1$.
By (\ref{result-1}),
we get that
\begin{equation}\label{result-8}
{\widetilde u}_2,\;\;{\widehat u}_2\in {\mathcal U}_2,
\end{equation}
and
\begin{equation}\label{result-9}
J_2(u_1,{\widetilde u}_2)\leq J_2(u_1,v_2)\;\;\mbox{and}\;\;
J_2(u_1,{\widehat u}_2)\leq J_2(u_1,v_2)\;\;\mbox{for each}\;\;v_2\in {\mathcal U}_2.
\end{equation}
For any $\lambda\in [0,1]$, by (\ref{Intro-2}) and (\ref{Intro-1}), we have that
\begin{eqnarray*}
&&J_2(u_1,\lambda {\widetilde u}_2+(1-\lambda){\widehat u}_2)-[\lambda J_2(u_1,{\widetilde u}_2)+(1-\lambda)J_2(u_1,{\widehat u}_2)]\\[3mm]
&=&\|y(T;y_0,u_1,\lambda {\widetilde u}_2+(1-\lambda){\widehat u}_2)-y^2_T\|_{L^2(G_2)}-\lambda\|y(T;y_0,u_1,{\widetilde u}_2)-y^2_T\|_{L^2(G_2)}\\[3mm]
&&-(1-\lambda)\|y(T;y_0,u_1,{\widehat u}_2)-y^2_T\|_{L^2(G_2)}\\[3mm]
&=&\|\lambda [y(T;y_0,u_1,{\widetilde u}_2)-y^2_T]+(1-\lambda)[y(T;y_0,u_1,{\widehat u}_2)-y^2_T]\|_{L^2(G_2)}\\[3mm]
&&-\lambda\|y(T;y_0,u_1,{\widetilde u}_2)-y^2_T\|_{L^2(G_2)}-(1-\lambda)\|y(T;y_0,u_1,{\widehat u}_2)-y^2_T\|_{L^2(G_2)}\\[3mm]
&\leq&0.
\end{eqnarray*}
This, along with (\ref{result-8}) and (\ref{result-9}), yields that
\begin{equation*}
\lambda {\widetilde u}_2+(1-\lambda){\widehat u}_2\in {\mathcal U}_2
\end{equation*}
and
\begin{equation*}
J_2(u_1,\lambda{\widetilde u}_2+(1-\lambda){\widehat u}_2)\leq J_2(u_1,v_2)\;\;\mbox{for each}\;\;v_2\in {\mathcal U}_2,
\end{equation*}
which indicate that $\lambda{\widetilde u}_2+(1-\lambda){\widehat u}_2\in \Phi_1 u_1$ (see (\ref{result-1})). Hence, $\Phi_1 u_1$
is a convex subset of ${\mathcal U}_2$.

{\it Step 4. Prove that Graph $\Phi$ is closed.}

It suffices to show that if $(u_{n,1},u_{n,2})\in {\mathcal U}_1\times {\mathcal U}_2,
{\widetilde u}_{n,1}\in \Phi_2 u_{n,2}$, ${\widetilde u}_{n,2}\in \Phi_1 u_{n,1}$,
$(u_{n,1},u_{n,2})\rightarrow (u_1,u_2)$ in $X$
and $({\widetilde u}_{n,1},{\widetilde u}_{n,2})\rightarrow ({\widetilde u}_1,{\widetilde u}_2)$  in
$X$, then
\begin{equation}\label{20180208-3}
(u_1,u_2)\in {\mathcal U}_1\times {\mathcal U}_2,\; {\widetilde u}_1\in \Phi_2 u_2\;\;
\mbox{and}\;\;{\widetilde u}_2\in \Phi_1 u_1.
\end{equation}
Indeed, on one hand, by (\ref{result-1}) and (\ref{result-2}), we can easily check that
\begin{equation}\label{result-10}
(u_1,u_2)\in {\mathcal U}_1\times {\mathcal U}_2,\; {\widetilde u}_1\in {\mathcal U}_1\;\;
\mbox{and}\;\;{\widetilde u}_2\in {\mathcal U}_2.
\end{equation}
On the other hand, according to ${\widetilde u}_{n,1}\in \Phi_2 u_{n,2}$, (\ref{result-2}) and (\ref{Intro-2}), it is obvious that for each $v_1\in {\mathcal U}_1$,
\begin{equation}\label{result-11}
\|y(T;y_0,g,{\widetilde u}_{n,1},u_{n,2})-y^1_T\|_{L^2(G_1)}\leq \|y(T;y_0,g,v_1,u_{n,2})-y^1_T\|_{L^2(G_1)}.
\end{equation}
Since $({\widetilde u}_{n,1},u_{n,2})\rightarrow ({\widetilde u}_1,u_2)$ weakly star in $(L^\infty(0,T;L^2(0,1)))^2$, by similar arguments as
those to get (\ref{result-6}), there exists a subsequence of $\{n\}_{n\geq1}$, still denoted by itself, so that
\begin{equation*}
\begin{array}{lll}
&(y(T;y_0,g,{\widetilde u}_{n,1},u_{n,2}),y(T;y_0,g,v_1,u_{n,2}))\\[3mm]
& \rightarrow (y(T;y_0,g,{\widetilde u}_1,u_2),y(T;y_0,g,v_1,u_2))
~\text{strongly in}~  (L^2(0,1))^2,
\end{array}
\end{equation*}
which, implies that
\begin{equation}\label{result-12}
    \begin{split}
        &(y(T;y_0,g,{\widetilde u}_{n,1},u_{n,2}),y(T;y_0,g,v_1,u_{n,2}))\\[3mm]
& \rightarrow (y(T;y_0,g,{\widetilde u}_1,u_2),y(T;y_0,g,v_1,u_2))
~\text{strongly in}~  (L^2(G_1))^2.
    \end{split}
\end{equation}
Passing to the limit for $n\rightarrow \infty$ in (\ref{result-11}), by (\ref{result-12}),
we get that for each $v_1\in {\mathcal U}_1$,
\begin{equation*}
\|y(T;y_0,g,{\widetilde u}_1,u_2)-y_T^1\|_{L^2(G_1)}\leq \|y(T;y_0,g,v_1,u_2)-y_T^1\|_{L^2(G_1)}.
\end{equation*}
This, together with (\ref{Intro-2}), (\ref{result-2}) and the second conclusion in (\ref{result-10}), implies that
${\widetilde u}_1\in \Phi_2 u_2$. Similarly, ${\widetilde u}_2\in \Phi_1 u_1$. Hence, (\ref{20180208-3}) follows.

{\it Step 5. Finish the proof.}

According to Steps 1-4 and Lemma~\ref{Intro-5}, there exists a pair of $(u_1^*,u_2^*)\in {\mathcal U}_1\times {\mathcal U}_2$
so that $(u_1^*,u_2^*)\in \Phi(u_1^*,u_2^*)$,
which, combined with (\ref{result-1})-(\ref{result-3}), indicates that
$(u_1^*,u_2^*)$ is a Stackelberg-Nash equilibrium of the problem {\bf(P1)}.

In summary, we end the proof of Theorem~\ref{Intro-3}.
\end{proof}


\subsection{Decomposition and Characterization of Stackelberg-Nash Equilibria}
In this subsection, we first give some conclusions that will be used later. By Theorem \ref{thm:obs-inq}, we have the observability inequality for adjoint equation (\ref{eq:adjoint}).
\begin{corollary}\label{co-obine}
Let $\alpha\in(0,2)$, $\mu\in(0,1)$ and $\omega\times E\subset I\times(0,T)$  be measurable subsets with positive measures. Then there exists a constant $C=C(T, I,\alpha,\omega,E,\mu)\geq1$ such that the solution of (\ref{eq:adjoint})  satisfies the following observability inequality: for any $z_T\in L^2(0,1)$,
\begin{equation}
  \label{eq:obs-inq1}
  \|z(0)\|\le C \int_E\|\chi_{\omega} z(t)\|dt.
\end{equation}
\end{corollary}

\begin{lemma}
\label{lemma-main}
Let $g\in L^\infty(0,T;L^2(0,1))$ be given. Assume $\omega,\omega_0$ are measurable subsets of $(0,1)$with positive measures. Let $\omega_0\subset G$, where $G$ is a measurable subset of $(0,1)$ and $\zeta(\cdot;\zeta_0,g,u)$ be the solution to the following system:
\begin{equation}
\label{model-zeta}
\left\{
\begin{array}{ll}
\zeta_t(x,t) = A\zeta(x,t) +\chi_\omega g+\chi_{\omega_0} u, & \left( x ,t\right) \in (0,1)\times(0,T),   \\[2mm]
\zeta(1,t)=BC_\alpha\big(\zeta(\cdot,t)\big) = 0, & t\in \left(0,T\right), \\[3mm]
\zeta\left(x, 0\right) =\zeta_{0}(x), &  x\in (0,1)\,,
\end{array}
\right.
\end{equation}
where $u\in \cU$ with
$$
\cU:=\big\{u\in L^\infty(0,T;L^2(0,1)):\|u\|_{L^\infty(0,T;L^2(0,1))}\leq M  \big\},
$$
with $M>0$. Consider the following problem:
\begin{equation}
\label{lemma-P}
\inf_{u\in \cU}\|\zeta(T;\zeta_0,g,u)-\zeta_T\|_{L^2(G)},
\end{equation}
where $\zeta_T\in L^2(0,1)$ is a given function. Then there exists $u^*\in \cU$ such that
$$
\|\zeta(T;\zeta_0,g,u^*)-\zeta_T\|_{L^2(G)}=\inf_{u\in \cU}\|\zeta(T;\zeta_0,g,u)-\zeta_T\|_{L^2(G)}.
$$
Moreover, if $\zeta(T;\zeta_0,g,u^*)\neq\zeta_T$ , then $u^*$ enjoys the following form
\begin{equation}
\label{lemma-u}
u^*(t)=M\dfrac{\chi_{\omega_0} z(t)}{\|\chi_{\omega_0} z(t)\|_{L^2(G)}},~\text{for a.e.}~t\in (0,T).
\end{equation}
where $z$ is a solution solving the following equation:
\begin{equation}
\label{BSDE}
\left\{
\begin{array}{ll}
z_t(x,t) +Az(x,t)= 0, & \left( x ,t\right) \in (0,1)\times(0,T),   \\[2mm]
z(1,t)=BC_\alpha\big(z(\cdot,t)\big) = 0, & t\in \left(0,T\right), \\[2mm]
z\left(x, T\right) =\zeta_T-\zeta(T;\zeta_0,g,u^*), &  x\in (0,1).
\end{array}
\right.
\end{equation}
\end{lemma}
\begin{proof}
We carry out this proof step by step.

{\it Step 1.} We show the existence of the problem (\ref{lemma-P}).

Note that $\zeta(\cdot;\zeta_0,g,u)=\zeta(\cdot;\zeta_0,g,0)+\zeta(\cdot;0,0,u)$, then the problem (\ref{lemma-P}) is equivalent to the
following problem:
$$
\inf_{u\in \cU}\|\zeta(T;0,0,u)-(\zeta_T-\zeta(T;\zeta_0,g,0))\|_{L^2(G)}.
$$
Let $\{u_n\}_{n\geq1}\subseteq\cU$ be the minimal sequence of problem (\ref{lemma-P}), i.e.,
\begin{equation}
\label{lemma-eq1}
\|\zeta(T;0,0,u_n)-(\zeta_T-\zeta(T;\zeta_0,g,0))\|_{L^2(G)}\rightarrow\inf_{u\in \cU}\|\zeta(T;0,0,u)-(\zeta_T-\zeta(T;\zeta_0,g,0))\|_{L^2(G)}.
\end{equation}
Since $\{u_n\}_{n\geq1}\subseteq\cU$, there exists a subsequence of $\{u_n\}_{n\geq1}$, still denoted by itself, and $u^*\in L^\infty(0,T;L^2(0,1))$ such that
$$
u_n\rightarrow u^*\quad\ \text{weakly star in}~\ L^\infty(0,T;L^2(0,1)).
$$
Note that $\zeta(\cdot; 0,0, u_n)$ is the solution of the following system
\begin{equation*}
\left\{
\begin{array}{ll}
\zeta(t; 0,0, u_n)_t(x) = A\zeta(t; 0,0, u_n)(x) +\chi_{\omega_0} u_n, & \left( x ,t\right) \in (0,1)\times(0,T),   \\[2mm]
\zeta(t; 0,0, u_n)(1)=BC_\alpha(\zeta(t; 0,0, u_n)(\cdot))
 = 0, & t\in \left(0,T\right), \\[2mm]
\zeta(0; 0,0, u_n)(x) =0, &  x\in (0,1).
\end{array}
\right.
\end{equation*}
By similar arguments as
those to get (\ref{result-6}), there exists a subsequence of $\{n\}_{n\geq1}$, still denoted by itself and $\zeta^*\in C([0,T];L^2(0,1))\cap L^2(0,T;H_a^1(0,1))$ such that
\begin{align}
\label{2.18.1}
\zeta(\cdot; 0,0, u_n)\rightarrow \zeta^*\;\;&\mbox{weakly in}\;\;C([0,T];L^2(0,1))\cap L^2(0,T;H_a^1(0,1))\nonumber\\
&~\text{and strongly in}~L^2(0,1).
\end{align}
Passing to the limit for $n\rightarrow\infty$ in (\ref{2.18.1}), we obtain that $\zeta^*$ is the solution of the following
system
$$
\left\{
\begin{array}{ll}
\zeta^*_t(x,t) = A\zeta^*(x,t) +\chi_{\omega_0} u^*, & \left( x ,t\right) \in (0,1)\times(0,T),   \\[2mm]
\zeta^*(1,t)=BC_\alpha\big(\zeta^*(\cdot,t)\big)
 = 0, & t\in \left(0,T\right), \\[2mm]
\zeta^*\left(x, 0\right) =0, &  x\in (0,1).
\end{array}
\right.
$$
This means that $\zeta^*=\zeta(\cdot; 0,0 ,u^*)$. Now, letting $n\rightarrow\infty$ in (\ref{lemma-eq1}) we obtain that
$$
\|\zeta(T;0,0,u^*)-(\zeta_T-\zeta(T;\zeta_0,g,0))\|_{L^2(G)}=\inf_{u\in \cU}\|\zeta(T;0,0,u)-(\zeta_T-\zeta(T;\zeta_0,g,0))\|_{L^2(G)}.
$$
This proves the existence of problem (\ref{lemma-P}). In other words, $u^*$ is the optimal control for the problem (\ref{lemma-P}).

{\it Step 2.} We characterize the optimal control $u^*$.

Denote $\hat{\zeta}_T=\zeta_T-\zeta(T;\zeta_0,g,0)$. Since $u^*$ is an optimal control, i.e., for each $u\in \cU$ we have
\begin{equation}
\label{lemma-eq2}
\|\zeta(T;0,0,u^*)-\hat{\zeta}_T\|_{L^2(G)}\leq\|\zeta(T;0,0,u)-\hat{\zeta}_T\|_{L^2(G)}.
\end{equation}
Set
$$
u_\lambda:= u^*+\lambda(u-u^*),\quad \lambda\in[0,1].
$$
It is obvious that $u_\lambda\in \cU$. Then, by (\ref{lemma-eq2}), we see
\begin{align*}
&\|\zeta(T;0,0,u^*)\|^2_{L^2(G)}-2\langle \zeta(T;0,0,u^*), \hat{\zeta}_T\rangle_{L^2(G)}\\
\leq& \|\zeta(T;0,0,u_\lambda)\|^2_{L^2(G)}-2\langle \zeta(T;0,0,u_\lambda), \hat{\zeta}_T\rangle_{L^2(G)}.
\end{align*}
After some calculations, the above inequality holds
$$
\|\zeta(T;0,0,u-u^*)\|_{L^2(G)}\lambda^2+2\lambda\langle \zeta(T;0,0,u-u^*),\zeta(T;0,0,u^*)- \hat{\zeta}_T\rangle_{L^2(G)}\geq0,
$$
which, implies that for all $u\in\cU$
\begin{equation}
\label{lemma-eq3}
\langle \zeta(T;0,0,u-u^*),\zeta(T;0,0,u^*)- \hat{\zeta}_T\rangle_{L^2(G)}\geq0.
\end{equation}
Noting that $\hat{\zeta}_T=\zeta_T-\zeta(T;\zeta_0,g,0)$, and plugging it into (\ref{BSDE}), then multiplying the first equation of (\ref{BSDE}) by $\zeta(\cdot;0,0,u-u^*)$ and integrating over $G\times(0,T)$, one deduces,
$$
-\langle \zeta(T;0,0,u-u^*),\zeta(T;0,0,u^*)- \hat{\zeta}_T\rangle_{L^2(G)} =\int_0^T \left\langle u(t)-u^*(t),\chi_{\omega_0} z(t)\right\rangle_{L^2(G)} dt.
$$
This, along with (\ref{lemma-eq3}) stands that
$$
\int_0^T \left\langle u(t)-u^*(t), \chi_{\omega_0} z(t)\right\rangle_{L^2(G)} dt\leq0, \text{for each}~ u\in\cU,
$$
which, implies
$$
\int_0^T\langle \chi_{\omega_0} z(t),u^*(t) \rangle_{L^2(G)} dt=\max_{u(\cdot)\in \cU} \int_0^T\langle \chi_{\omega_0} z(t),u(t) \rangle_{L^2(G)} dt.
$$
This is equivalent to the following condition:
\begin{equation}
\label{lemma-eq4}
\langle \chi_{\omega_0} z(t),u^*(t) \rangle_{L^2(G)} =\max_{u(\cdot)\in\cU}  \langle \chi_{\omega_0} z(t),u(t) \rangle_{L^2(G)}, ~\text{for}~a.e.\,\,t\in (0,T).
\end{equation}
By the condition $\zeta(T;\zeta_0,g,u^*)\neq \zeta_T$ and noting that $\hat{\zeta}_T-\zeta(T;0,0,u^*)=\zeta_T-\zeta(T;\zeta_0,g,0)-\zeta(T;0,0,u^*)=\zeta_T-\zeta(T;\zeta_0,g,u^*)$, it holds $\hat{\zeta}_T-\zeta(T;0,0,u^*)\neq0$. This, together with (\ref{eq:obs-inq1}) in Corollary \ref{co-obine},  implies
$$
\|\chi_{\omega_0} z(t)\|\neq0,~\text{for}~a.e. \,\,t\in(0,T).
$$
By (\ref{lemma-eq4}), we see
$$
\langle \chi_{\omega_0} z(t),u^*(t) \rangle_{L^2(G)} \leq M\|\chi_{\omega_0} z(t)\|.
$$
Denote
$$
u_0(t)=M\frac{\chi_{\omega_0} z(t)}{\|\chi_{\omega_0} z(t)\|}~\text{for a.e.}~t\in (0,T).
$$
It implies $u_0\in\cU$. Also, we have
$$
\langle \chi_{\omega_0} z(t),u_0(t) \rangle_{L^2(G)}=\left\langle \chi_{\omega_0} z(t),M\frac{\chi_{\omega_0} z(t)}{\|\chi_{\omega_0} z(t)\|} \right\rangle_{L^2(G)}=M\|\chi_{\omega_0} z(t)\|_{L^2(G)},~\text{for a.e.}~t\in (0,T),
$$
which, along with (\ref{lemma-eq4}) again, it holds
$$
\langle \chi_{\omega_0} z(t),u^*(t) \rangle_{L^2(G)}=M\|\chi_{\omega_0} z(t)\|_{L^2(G)}~\text{for a.e.}~t\in (0,T).
$$
Therefore, we have (\ref{lemma-u}). This completes the proof.
\end{proof}

Next, we shall obtain that the set of Stackelberg-Nash equilibria breaks down into three disjoint parts.

Let N be the set of Stackelberg-Nash equilibria of problem {\bf(P1)} with respect to the system (\ref{Intro-1}), and let
\begin{equation}
\label{eq-N}
\begin{array}{lll}
&&N_0=\big\{ (u_1^*,u_2^*)\in N:y(T,y_0,g,u_1^*,u_2^*)\neq y_T^1,y(T,y_0,g,u_1^*,u_2^*)\neq y_T^2 \big\},\\[3mm]
&&N_1=\big\{ (u_1^*,u_2^*)\in N:y(T,y_0,g,u_1^*,u_2^*)= y_T^1 \big\},\\[3mm]
&&N_2=\big\{ (u_1^*,u_2^*)\in N:y(T,y_0,g,u_1^*,u_2^*)= y_T^2 \big\}.\\[3mm]
\end{array}
\end{equation}
Clearly, $N_0,N_1,N_2$ are three disjoint subsets of $N$ and $N = N_0 \cup N_1 \cup N_2$.
\begin{proposition}
\label{pro1}
Let $g\in L^\infty(0,T;L^2(0,1))$ be given and $(u_1^*,u_2^*)\in N$. Then either
$$
\|u_1^*\|_{L^2(G_1)}=M_1,~\text{for a.e.}~ t\in(0,T),
$$
or
$$
\|u_2^*\|_{L^2(G_2)}=M_2,~\text{for a.e.}~ t\in(0,T).
$$

\end{proposition}

\begin{proof}
Let $(u_1^*,u_2^*)\in N$. Then we have
$$
\begin{array}{lll}
J_1(u_1^*,u_2^*)\leq J_1(u_1,u_2^*),\,\, \forall\,u_1\in\cU_1,\\[3mm]
J_2(u_1^*,u_2^*)\leq J_1(u_1^*,u_2),\,\, \forall\,u_2\in\cU_2.
\end{array}
$$
Similar to the proof of Lemma \ref{lemma-main}, we obtain:

Case 1: if $y(T,y_0,g,u_1^*,u_2^*)\neq y_T^1$, then
  $$
  u^*_1(t)=M_1\dfrac{\chi_{\omega_1}z(t)}{\|\chi_{\omega_1}z(t)\|_{L^2(G_1)}},~\text{for a.e.}~t\in (0,T),
  $$
  where $z$ is a solution solving the following equation:
$$
\left\{
\begin{array}{ll}
z_t(x,t) +Az(x,t)= 0, & \left( x ,t\right) \in (0,1)\times(0,T),   \\[2mm]
z(1,t)=BC_\alpha\big(z(\cdot,t)\big) = 0, & t\in \left(0,T\right), \\[3mm]
z\left(x, T\right) =y_T^1-y(T,y_0,g,u_1^*,u_2^*), &  x\in (0,1).
\end{array}
\right.
$$
One can easily check that
$$
\|u_1^*\|_{L^2(G_1)}=M_1,~\text{for a.e.}~ t\in(0,T).
$$

Case 2: if $y(T,y_0,g,u_1^*,u_2^*)\neq y_T^2$, then
  $$
  u^*_2(t)=M_2\dfrac{\chi_{\omega_2}z(t)}{\|\chi_{\omega_2}z(t)\|_{L^2(G_2)}},~\text{for a.e.}~t\in (0,T),
  $$
  where $z$ is a solution solving the following equation:
  $$
\left\{
\begin{array}{ll}
z_t(x,t) +Az(x,t)= 0, & \left( x ,t\right) \in (0,1)\times(0,T),   \\[2mm]
z(1,t)=BC_\alpha\big(z(\cdot,t)\big)  = 0, & t\in \left(0,T\right), \\[3mm]
z\left(x, T\right) =y_T^2-y(T,y_0,g,u_1^*,u_2^*), &  x\in (0,1).
\end{array}
\right.
$$
One can easily check that
$$
\|u_2^*\|_{L^2(G_2)}=M_2,~\text{for a.e.}~ t\in(0,T).
$$
The proof is completed.
\end{proof}

\begin{remark}\label{N0}
    If Stackelberg-Nash equilibria $(u_1^*,u_2^*)\in N_0$, we have
    \begin{align*}
        u^*_1(t)=M_1\dfrac{\chi_{\omega_1}z(t)}{\|\chi_{\omega_1}z(t)\|_{L^2(G_1)}},~\text{for a.e.}~t\in (0,T),\\
u^*_2(t)=M_2\dfrac{\chi_{\omega_2}z(t)}{\|\chi_{\omega_2}z(t)\|_{L^2(G_2)}},~\text{for a.e.}~t\in (0,T).
    \end{align*}
\end{remark}

\begin{proposition}
\label{pro2}
Let $(u_1^*,u_2^*)\in N_1$, i.e.,
\begin{equation}
\label{pro2-eq-1}
u^*_2(t)=M_2\dfrac{\chi_{\omega_2}z(t)}{\|\chi_{\omega_2}z(t)\|_{L^2(G_2)}},~\text{for a.e.}~t\in (0,T),
\end{equation}
where $z$ is a solution for adjoint equation (\ref{BSDE}) with the terminal condition $z(T)=y_T^2-y(T;y_0,g,u_1,u_2^*)$, $u_1\in\cU_1$.

Denote

$$
A=
\left\{u_1\in\cU_1\ \left|\
\begin{array}{lll}
y(T,y_0,g,u_1,u_2^*)=y_T^1,\\[2mm]
\|y(T;y_0,g,u_1,u_2^*)-y^2_T\|_{L^2(G_2)}\leq\|y(T;y_0,g,u_1,u_2)-y^2_T\|_{L^2(G_2)},\\[2mm]
\forall\,u_2\in\cU_2.
\end{array}\right.
\right\},
$$
and
$$
B=
\left\{
\begin{array}{lll}
u_1\in\cU_1: y(T,y_0,g,u_1,u_2^*)=y_T^1
\end{array}
\right\}.
$$
Then, $A=B$.

\end{proposition}

\begin{proof}
It is clear that $A\subseteq B$. We only need to show $B\subseteq A$.

Let $u_1\in B$. Then there exists an element $u_2\in\cU_2$ such that the following problem holds:
\begin{equation}
\label{pro2-P}
\left\{
\begin{array}{lll}
y(T,y_0,g,u_1,u_2)=y_T^1,\\[3mm]
\|y(T;y_0,g,u_1,u_2)-y^2_T\|_{L^2(G_2)}\leq\|y(T;y_0,g,u_1,w_2)-y^2_T\|_{L^2(G_2)},\,\,\forall\,w_2\in\cU_2.
\end{array}
\right.
\end{equation}
For every $w_2\in\cU_2$ and $\lambda\in(0,1)$, set $$w_\lambda=u_2+\lambda(w_2-u_2),$$ then $w_\lambda\in\cU_2$. Then for all $\lambda\in(0,1)$, we have$$
\|y(T;y_0,g,u_1,u_2)-y^2_T\|_{L^2(G_2)}\leq\|y(T;y_0,u_1,w_\lambda)-y^2_T\|_{L^2(G_2)},
$$
which shows that for all $\lambda\in(0,1)$,
$$
\|y(T;y_0,g,u_1,u_2)-y^2_T\|^2_{L^2(G_2)}\leq\|y(T;y_0,g,u_1,w_\lambda)-y^2_T\|^2_{L^2(G_2)}.
$$
This implies that for all $\lambda\in(0,1)$,
\begin{align*}
&2\langle y(T;y_0,g,u_1,w_\lambda)-y(T;y_0,g,u_1,u_2),y_T^2  \rangle_{L^2(G_2)}\\
&\leq\langle y(T;y_0,g,u_1,w_\lambda)-y(T;y_0,g,u_1,u_2),y(T;y_0,u_1,w_\lambda)+y(T;y_0,g,u_1,u_2)  \rangle_{L^2(G_2)}.
\end{align*}
In other words, for all $\lambda\in(0,1)$, we have
\begin{align*}
&2\lambda\langle y(T;0,0,0,w_2-u_2),y_T^2  \rangle_{L^2(G_2)}\\
&\leq\langle \lambda y(T;0,0,0,w_2-u_2),2y(T;y_0,g,u_1,u_2)+\lambda y(T;0,0,0,w_2-u_2)  \rangle_{L^2(G_2)}.
\end{align*}
Letting $\lambda\rightarrow0^+$, it holds
$$
\langle y(T;0,0,0,w_2-u_2),y_T^2  \rangle_{L^2(G_2)}\leq \langle y(T;0,0,0,w_2-u_2),y(T;y_0,g,u_1,u_2)\rangle_{L^2(G_2)}.
$$
This shows for all $w_2\in\cU_2$, it stands
$$
\langle y(T;0,0,0,w_2-u_2),y_T^2- y(T;y_0,g,u_1,u_2)\rangle_{L^2(G_2)}\leq0.
$$
In other words, for all $w_2\in\cU_2$, we have
$$
\langle y(T;0,0,0,w_2-u_2),y_T^2- y_T^1\rangle_{L^2(G_2)}\leq0.
$$
From these, we obtain
$$
\left\langle \int^T_0e^{A(T-t)}\chi_{\omega_2}(w_2-u_2)dt,y_T^2- y_T^1\right\rangle_{L^2(G_2)}\leq0.
$$
This implies that for all $w_2\in\cU_2$,
$$
\int^T_0\left\langle\chi_{\omega_2}(w_2-u_2),e^{A(T-t)}(y_T^2- y_T^1)\right\rangle_{L^2(G_2)}dt\leq0.
$$
Therefore, we obtain for all $w_2\in\cU_2$,
$$
\left\langle(w_2-u_2),\chi_{\omega_2}e^{A(T-t)}(y_T^2- y_T^1)\right\rangle_{L^2(0,T;L^2(G_2))}\leq0.
$$
This is equivalent to the following condition: for $a.e.\,\,t\in(0,T)$
$$
\langle \chi_{\omega_2} e^{A(T-t)}(y_T^2- y_T^1),u_2(t) \rangle_{L^2(G_2)} =\max_{w_2(\cdot)\in\cU_2}  \langle \chi_{\omega_2} e^{A(T-t)}(y_T^2- y_T^1),w_2(t) \rangle_{L^2(G_2)}.
$$
By the same line as the proof of Lemma \ref{lemma-main}, it holds
$$
u^*_2(t)=M_2\dfrac{\chi_{\omega_2}z(t)}{\|\chi_{\omega_2}z(t)\|_{L^2(G_2)}},~\text{for a.e.}~t\in (0,T),
$$
which, together with (\ref{pro2-eq-1}), shows
$$
\begin{array}{lll}
y(T,y_0,g,u_1,u_2^*)=y_T^1,\\[3mm]
\|y(T;y_0,g,u_1,u_2^*)-y^2_T\|_{L^2(G_2)}\leq\|y(T;y_0,g,u_1,w_2)-y^2_T\|_{L^2(G_2)},\,\,\forall\,w_2\in\cU_2.
\end{array}
$$
This implies $u_1\in A$. Hence, we have $B\subseteq A$. The proof is completed.
\end{proof}

\begin{proposition}
\label{pro3}
Let $(u_1^*,u_2^*)\in N_2$, i.e.,
\begin{equation}
\label{pro3-eq-1}
u^*_1(t)=M_1\dfrac{\chi_{\omega_1}z(t)}{\|\chi_{\omega_1}z(t)\|_{L^2(G_1)}},~\text{for a.e.}~t\in (0,T),
\end{equation}
where $z$ is a solution for adjoint equation (\ref{BSDE}) with the terminal condition $z(T)=y_T^1-y(T;y_0,g,u_1^*,u_2)$, $u_2\in\cU_2$.

Denote
$$
\bar{A}=
\left\{u_2\in\cU_2\ \left|\
\begin{array}{lll}
y(T,y_0,g,u_1^*,u_2)=y_T^2,\\[2mm]
\|y(T;y_0,g,u_1^*,u_2)-y^1_T\|_{L^2(G_1)}\leq\|y(T;y_0,g,u_1,u_2)-y^1_T\|_{L^2(G_1)},\\[2mm]
\forall\,u_1\in\cU_1.
\end{array}\right.
\right\},
$$
and
$$
\bar{B}=
\left\{
\begin{array}{lll}
u_1\in\cU_1: y(T,y_0,g,u_1^*,u_2)=y_T^2
\end{array}
\right\}.
$$
Then, $\bar{A}=\bar{B}$.

\end{proposition}

\begin{proof}
Since the proof is similar to that of Proposition \ref{pro2}, we here omit the detail.
\end{proof}

\begin{remark}
\label{remark-N1-N2}
Under the assumption of Proposition \ref{pro2}, if $u_1\in A$, then $u_1\in\cU_1$, and
$$
\begin{array}{lll}
&&y(T,y_0,g,u_1,u_2^*)=y_T^1,\\[3mm]
&&\|y(T;y_0,g,u_1,u_2^*)-y^2_T\|_{L^2(G_2)}\leq\|y(T;y_0,g,u_1,u_2)-y^2_T\|_{L^2(G_2)},\,\,\forall\,u_2\in\cU_2.
\end{array}
$$
Thus,
$$
\begin{array}{lll}
&&\|y(T;y_0,g,u_1,u_2^*)-y^1_T\|_{L^2(G_1)}=0\leq\|y(T;y_0,g,w_1,u_2^*)-y^1_T\|_{L^2(G_1)},\,\,\forall\,w_1\in\cU_1,\\[3mm]
&&\|y(T;y_0,g,u_1,u_2^*)-y^2_T\|_{L^2(G_2)}\leq\|y(T;y_0,g,u_1,u_2)-y^2_T\|_{L^2(G_2)},\,\,\forall\,u_2\in\cU_2.
\end{array}
$$
These imply $(u_1,u_2^*)\in N$. i.e.,
$$
A=\{ u_1\in\cU_1: (u_1,u_2^*)\in N, y(T,y_0,g,u_1,u_2^*)=y_T^1 \}.
$$
Therefore, Proposition \ref{pro2} tells us: if
$$
u^*_2=M_2\dfrac{\chi_{\omega_2}e^{A(T-\cdot)}(y_T^2- y_T^1)}{\|\chi_{\omega_2}e^{A(T-\cdot)}(y_T^2- y_T^1)\|_{L^2(G_2)}},
$$
then
$$
A=\{ u_1\in\cU_1: (u_1,u_2^*)\in N, y(T,y_0,g,u_1,u_2^*)=y_T^1 \}=\{ u_1\in\cU_1: y(T,y_0,g,u_1,u_2^*)=y_T^1 \}.
$$
Similarly, Proposition \ref{pro3} tells us: if
$$
u^*_1=M_1\dfrac{\chi_{\omega_1}e^{A(T-\cdot)}(y_T^1- y_T^2)}{\|\chi_{\omega_1}e^{A(T-\cdot)}(y_T^1- y_T^2)\|_{L^2(G_1)}},
$$
then
$$
\bar{A}=\{ u_2\in\cU_2: (u_1^*,u_2)\in N, y(T,y_0,g,u_1^*,u_2)=y_T^2 \}=\{ u_2\in\cU_2: y(T,y_0,g,u_1^*,u_2)=y_T^2 \}.
$$

\end{remark}
From Remark \ref{remark-N1-N2}, we obtain the following further characterizations of $N_1$ and $N_2$.

\begin{proposition}
\label{pro4}
Let $g\in L^\infty(0,T;L^2(0,1))$ be given. Then we have the following characterizations:
$$
N_1=
\left\{(u_1^*,u_2^*)\in(\cU_1\times\cU_2)\ \left|\
\begin{array}{lll}
y(T,y_0,g,u_1^*,u_2^*)=y_T^1,\\[2mm]
u^*_2=M_2\dfrac{\chi_{\omega_2}e^{A(T-\cdot)}(y_T^2- y_T^1)}{\|\chi_{\omega_2}e^{A(T-\cdot)}(y_T^2- y_T^1)\|_{L^2(G_2)}}.
\end{array}\right.
\right\}
$$
and
$$
N_2=
\left\{(u_1^*,u_2^*)\in(\cU_1\times\cU_2)\ \left|\
\begin{array}{lll} y(T,y_0,g,u_1^*,u_2^*)=y_T^2,\\[2mm]
u^*_1=M_1\dfrac{\chi_{\omega_1}e^{A(T-\cdot)}(y_T^1- y_T^2)}{\|\chi_{\omega_1}e^{A(T-\cdot)}(y_T^1- y_T^2)\|_{L^2(G_1)}}.
\end{array}\right.
\right\}.
$$
\end{proposition}

In conclusion, Theorem \ref{Intro-4} is a direct consequence of (\ref{eq-N}), Proposition \ref{pro1} and Proposition \ref{pro4}.

\section{The norm optimal control problem}
In this section, we discuss the norm optimal control problem {\bf(P2)}. At first, we study the existence, i.e., Lemma~\ref{lemma-g-exist}. Then the leader control $g$ is characterized
in terms of the solution of a variational problem, i.e., Theorem \ref{9.15.1}.

\begin{proof}[\textbf{Proof of Lemma~\ref{lemma-g-exist}}]

 Let $\{g_n\}_{n\geq 1}\subset\cU_0$ be a minimizing sequence of problem {\bf(P2)} corresponding a  Stackelberg-Nash equilibrium sequence $(u_{1,n}^*(g_n),u_{2,n}^*(g_n))\in N=N_0\cup N_1\cup N_2, N\subset \cU_1\times\cU_2$ such
that
$$
\|g_n\|_{L^\infty(0,T;L^2(0,1))}\rightarrow\inf_{g\in \cU_0}\|g\|_{L^\infty(0,T;L^2(0,1))}.
$$
Since $\{g_n\}_{n=1}^\infty\subset L^\infty(0,T;L^2(0,1))$ is bounded, then there exists a subsequence, still denoted by itself, and $g^*\in L^\infty(0,T; L^2(0,1))$ such that
\begin{equation}\label{5.15.2}
    g_n\rightarrow g^* \mbox{ weakly star in } L^\infty(0,T;L^2(\Omega))
\end{equation}
with
\begin{equation}\label{5.15.1}
    \|g^*\|_{L^\infty(0,T;L^2(\Omega))}\leq \inf_{g\in\cU_0}\|g\|_{L^\infty(0,T;L^2(0,1))}.
\end{equation}
Note that $\|u_1^*(g_n)\|_{L^\infty(0,T;L^2(0,1))}\leq M_1, \|u_2^*(g_n)\|_{L^\infty(0,T;L^2(0,1))}\leq M_2$, then there exists a subsequence of $\{(u_1^*(g_n), u_2^*(g_n))\}_{n=1}^\infty \subset (L^\infty(0,T;L^2(0,1)))^2$, still denoted by itself, and $(u_1^*,u_2^*)\in (L^\infty(0,T;L^2(0,1)))^2$, such that
\begin{equation}\label{5.15.3}
    u_1^*(g_n)\rightarrow u_1^*, \mbox{and } u_2^*(g_n)\rightarrow u_2^* \mbox{ weakly star in } L^\infty(0,T; L^2(0,1)).
\end{equation}
According to
$y(T; y_0, g_n, u_1^*(g_n), u_2^*(g_n))=0$, we get $y(T; y_0, g^*,u_1^*,u_2^*)=0$ by \eqref{5.15.2} and \eqref{5.15.3}. By \eqref{5.15.1}, we only need to show $u_1^*=u_1^*(g^*), u_2^*=u_2^*(g^*)$.

It is obviously that there exists a subsequence of $\{(u_{1,n}^*(g_n), u_{2,n}^*(g_n))\}_{n=1}^\infty$, by extracting a subsequence (the subsequence is also a minimizing sequence of Problem {\bf(P2)}), still denoted by itself, such that the sequence  $\{(u_{1,n}^*(g_n), u_{2,n}^*(g_n))\}_{n=1}^\infty$ is contained in $N_0$, $N_1$ or $N_2$. This implies the following three cases:

Case $1$. If $(u_{1,n}^*(g_n),u_{2,n}^*(g_n))\in N_0$, i.e., $y_T^1\neq 0\neq y_T^2$, by Remark \ref{N0} and $(A_2)$, we have
\begin{align*}
u^*_{1,n}(g_n)=M_1\dfrac{\chi_{\omega_1}z_{1,n}}{\|\chi_{\omega_1}z_{1,n}\|_{L^2(G_1)}},\quad
u^*_{2,n}(g_n)=M_2\dfrac{\chi_{\omega_2}z_{2,n}}{\|\chi_{\omega_2}z_{2,n}\|_{L^2(G_2)}},
 \end{align*}  
where $z_{i,n}$ are the solutions of following systems
\begin{equation}\label{model77}
\left\{
\begin{array}{ll}
z_{i,n,t}(,t) +Az_{i,n}(x,t)= 0, & \left( x ,t\right) \in (0,1)\times(0,T),   \\[2mm]
z_{i,n}(1,t)=BC_\alpha\big(z_{i,n}(\cdot,t)\big)
= 0, & t\in \left(0,T\right), \\[3mm]
z_{i,n}\left( T\right) =y_T^i, &  x\in (0,1),
\end{array}
\right.
\end{equation}
for $i=1,2$, respectively.

Note that $z_{i,n}$ are independent on $n$, then $u_{1,n}^*(g_n)=u_{1}^*, u_{2,n}^*(g_n)=u_2^*$, i.e., $$
u^*_{1,n}(g_n)=M_1\dfrac{\chi_{\omega_1}z_{1}}{\|\chi_{\omega_1}z_{1}\|_{L^2(G_1)}},\quad\
u^*_{2,n}(g_n)=M_2\dfrac{\chi_{\omega_2}z_{2}}{\|\chi_{\omega_2}z_{2}\|_{L^2(G_2)}},
    $$
    where $(i=1,2)$
\begin{equation*}
\left\{
\begin{array}{ll}
z_{i,t}(x,t) +Az_{i}(x,t)= 0, & \left( x ,t\right) \in (0,1)\times(0,T),   \\[2mm]
z_{i}(1,t)=BC_\alpha\big(z_{i}(\cdot,t)\big)
 = 0, & t\in \left(0,T\right), \\[3mm]
z_{i}\left( T\right) =y_T^i, &  x\in (0,1).
\end{array}
\right.
\end{equation*}
Hence,
\begin{equation*}
 u^*_{1}(g^*)=M_1\dfrac{\chi_{\omega_1}z_{1}}{\|\chi_{\omega_1}z_{1}\|_{L^2(G_1)}},\quad\
u^*_{2}(g^*)=M_2\dfrac{\chi_{\omega_2}z_{2}}{\|\chi_{\omega_2}z_{2}\|_{L^2(G_2)}}
\end{equation*}
by \eqref{5.15.3} and Remark \ref{N0}.

Case $2$. If $(u_{1,n}^*(g_n),u_{2,n}^*(g_n))\in N_1$, i.e., $y_T^1=0$, by Proposition \ref{pro2} and assumption $(A_2)$, we have
$$
u^*_{2,n}(g_n)=M_2\dfrac{\chi_{\omega_2}z_{2,n}}{\|\chi_{\omega_2}z_{2,n}\|_{L^2(G_2)}},
    $$
where $z_{2,n}$ is the solution of following system
\begin{equation*}
\left\{
\begin{array}{ll}
z_{2,n,t}(x,t) +Az_{2,n}(x,t)= 0, & \left( x ,t\right) \in (0,1)\times(0,T),   \\[2mm]
z_{2,n}(1,t)=BC_\alpha\big(z_{2,n}(\cdot,t)\big)
= 0, & t\in \left(0,T\right), \\[3mm]
z_{2,n}\left(x, T\right) =y_T^2, &  x\in (0,1).
\end{array}
\right.
\end{equation*}
Similar to the proof of Case $1$, we have
$$
u^*_{2,n}(g_n)=M_2\dfrac{\chi_{\omega_2}z_{2}}{\|\chi_{\omega_2}z_{2}\|_{L^2(G_2)}}= u^*_{2}(g^*),
    $$
where $z_2$ is the solution of following system
\begin{equation*}
\left\{
\begin{array}{ll}
z_{2,t}(x,t) +Az_2(x,t)= 0, & \left( x ,t\right) \in (0,1)\times(0,T),   \\[2mm]
z_2(1,t)=BC_\alpha\big(z_{2}(\cdot,t)\big)
 = 0, & t\in \left(0,T\right), \\[3mm]
z_2\left(x, T\right) =y_T^2, &  x\in (0,1).
\end{array}
\right.
\end{equation*}
By Remark \ref{remark-N1-N2} and $y(T;y_0,g^*,u_1^*,u_2^*(g^*))=0=y_T^1$, we have $(u^*_{1},u^*_{2}(g^*))\in N_1$, i.e., $u_1^*=u_1^*(g^*)$.

Case $3$. If $(u_{1,n}^*(g_n),u_{2,n}^*(g_n))\in N_2$, i.e., $y_T^2=0$, the rest proof is similar to Case $2$, then we also have $u_1^*=u_1^*(g^*), u_2^*=u_2^*(g^*)$.

This completes the proof.
\end{proof}

\begin{remark}
 From the proof of Lemma~\ref{lemma-g-exist}, we see that under the assumption $(A_2)$, if there exists a Stackelberg-Nash equilibrium $(u_1^*(g),u_2^*(g))\in N$, then $(u_1^*,u_2^*)\in N_0$ or $(u_1^*(g),u_2^*)\in N_1$ or $(u_1^*,u_2^*(g))\in N_2$.
\end{remark}
\begin{remark}
Unfortunately, due to the influences of the functions $u^*_1(g^*)$ and $u^*_2(g^*)$ with respect to $g^*$, where $g^*$ is an optimal control of problem {\bf (P2)}, the confirmation of the bang-bang property of the optimal control $g^*$ remains elusive, thereby preventing the attainment of the uniqueness of the norm optimal control of problem {\bf (P2)}.
\end{remark}

\begin{lemma}\label{9.18.3}
		For any  $T\in\mathbb{R}^+$, let $ Y_T$ and $Z_T$ be defined as \eqref{9.18.1} and  \eqref{9.18.2} respectively. Then
		\begin{equation}\label{9.18.4}
		Y_T=Z_T.
		\end{equation}
	\end{lemma}
	
	\begin{proof} The proof will be accomplished by following two  steps.
		
		{\it Step 1: Show that $Y_T\subset Z_T$.}  	
		
		Indeed, for any  $\psi\in Y_T$ and noting the embedding $H_\alpha^1(0,1)\hookrightarrow L^2(0,1)$ is compact and similar to the proof of (\ref{result-5:7}) and (\ref{result-5:73}), there exists $\{z_{T,n}\}_{n\geq 1}\subset L^2(0,1)$ such that
\begin{equation}\label{9.18.8}
		 \chi_\omega z(\cdot;z_{T,n})\rightarrow \psi \mbox{ strongly in } L^1(0,T; L^2(0,1)),
		\end{equation}
		where $z(\cdot;z_{T,n})$ is the solution of system \eqref{eq:adjoint} with the terminal  data $z_{T,n}\in L^2(0,1)$.
		Hence
		\begin{equation}\label{9.18.9}
		\|\chi_\omega z(\cdot; z_{T,n})\|_{L^1(0,T; L^2(0,1))}\leq C.
		\end{equation}
		Choose  $\{T_n\}_{n\geq 1}$ with  $T_n\uparrow T$ and $T_n\geq \frac{T}{2}$ for all $n\in\mathbb{N}$.  %
We consider the following three cases.
		
		{\it Case a).}  For $T_1$, there exists a constant $C_1=C(C_O,T_2)>0$ such that
		\begin{equation*}
		\|z(T_2; z_{T,n})\|\leq C_1  \|\chi_\omega z(\cdot; z_{T,n})\|_{L^1(T_2,T; L^2(0,1))}\leq C_1C,
		\end{equation*}
		where we used the observability inequality \eqref{eq:obs-inq1}. Hence there exists a subsequence $\{z_{T,1n}\}_{n\geq 1}$ of $\{z_{T,n}\}_{n\geq 1}$ and $\hat z_{T,1}\in L^2(0,1)$ such that
		\begin{equation}\label{9.18.10}
		z(T_2; z_{T,1n})\rightarrow \hat z_{T,1} \mbox{ weakly in } L^2(0,1).
		\end{equation}
  
		Denote $z_{1n}=z(\cdot; z_{T,1n})$. Then  $z_{1n}$ is the solution to the following system
		\begin{equation*}
		\left\{\!\!
\begin{array}{ll}
z_{1n,t}(x,t) +Az_{1n}(x,t)= 0, & \left( x ,t\right) \in (0,1)\times(0,T_2),   \\[2mm]
z_{1n}(1,t)=BC_\alpha\big(z_{1n}(\cdot,t)\big)
 = 0, & t\in \left(0,T_2\right), \\[3mm]
z_{1n}\left(x, T\right) =z_{T,1n}, &  x\in (0,1)\,.
\end{array}
\right.
		\end{equation*}
		Let $\hat z_1\in L^1(0,T_2; L^2(\Om))\cap C([0,T]; L^2(\Om))$ be the solution to the following system
		\begin{equation*}
		\begin{cases}
		\partial_t\hat z_1(x,t)+A \hat z_1(x,t)=0, &\left( x ,t\right) \in (0,1)\times(0,T_2),\\
  \hat z_1(1,t)=BC_\alpha\big(\hat z_1(\cdot,t)\big)
 = 0, & t\in \left(0,T_2\right),\\
		\hat z_1(T_2)=\hat z_{T,1}, & x\in (0,1).
		\end{cases}
		\end{equation*}
Similar to the proof of (\ref{9.18.8}), by  \eqref{9.18.10}, for every  $\delta\in (\frac{T}{2}, T)$, we have
		\begin{equation*}
		\chi_\omega z_{1n}\rightarrow \chi_\omega\hat z_1 \mbox{ strongly in } L^1(0,T_2; L^2(0,1))\cap C([0,T_2-\delta); L^2(0,1)),
		\end{equation*}
		and
		\begin{equation}\label{9.18.11}
		\chi_\omega z_{1n}\rightarrow \chi_\omega\hat z_1 \mbox{ strongly in } L^1(0,T_1; L^2(0,1))\cap C([0,T_1]; L^2(0,1)).
		\end{equation}
		This, together with \eqref{9.18.8}, implies that
\begin{equation}\label{9.18.13}
		\psi=\chi_\omega \hat z_1 \mbox{ in } L^1(0,T_1; L^2(0,1)),
		\end{equation}
		where $\hat z_1(\cdot)=e^{A (T_1-\cdot)}z_{T,1}$ and $z_{T,1}=e^{A (T_2-T_1)}\hat z_{T,1}$.
		
		{\it Case b).}  Along the same way as case a), there exist a subsequence $\{z_{T,2n}\}$ of $\{z_{T,1n}\}$ and  $\hat z_2\in L^1(0,T_3;L^2(0,1))\cap C([0,T_3]; L^2(0,1))$ such that
		\begin{equation*}
		\begin{cases}
		\partial_t\hat z_2(x,t)+A \hat z_2(x,t)=0, &\left( x ,t\right) \in (0,1)\times(0,T_3),\\
 \hat z_2 (1,t)=BC_\alpha\big(\hat z_2(\cdot,t)\big)
 = 0, & t\in \left(0,T_3\right),\\
		\hat z_2(T_3)=\hat z_{T,2}, & x\in (0,1),
		\end{cases}
		\end{equation*}
		and
		\begin{equation*}
\chi_\omega z_{2n}=\chi_\omega z(\cdot; z_{T,2n})\rightarrow \chi_\omega\hat z_2 \mbox{ strongly in } L^1(0,T_2; L^2(0,1))\cap C([0,T_2]; L^2(0,1)).
		\end{equation*}
		These combine with \eqref{9.18.8}, \eqref{9.18.11} and \eqref{9.18.13}  imply that
		\begin{equation*}
		\chi_\omega\hat z_2|_{[0,T_1]}=\chi_\omega \hat z_1,
		\end{equation*}
		and
		\begin{equation*}
		\psi=\chi_\omega\hat z_2 \mbox{ in } L^1(0,T_2; L^2(0,1)).
		\end{equation*}
		where $\hat z_2(\cdot)=e^{A (T_2-\cdot)}z_{T,2}$ and $z_{T,2}=e^{A(T_3-T_2)}\hat z_{T,2}$.
		
		{\it Case c).}  Similarly to a) and b), for every $k\in\mathbb{N}$, we can find a subsequence $\{z_{T,kn}\}_{n\geq 1}$ of $\{z_{T,k-1,n}\}_{n\geq1}$ such that
		\begin{itemize}
			\item $ \hat z_k\in L^1(0,T_{k+1}; L^2(0,1))\cap C([0,T_{k+1}];L^2(0,1))$,
			
			\item $ \hat z_{k+1}|_{[0,T_{k}]}=\hat z_{k}$,
			
			\item $ \hat z_k$ satisfies
			\begin{equation*}
		\begin{cases}
		\partial_t\hat z_k(x,t)+A \hat z_k(x,t)=0, &\left( x ,t\right) \in (0,1)\times(0,T_{k+1}),\\
 \hat z_k (1,t)=BC_\alpha\big(\hat z_k(\cdot,t)\big)
 = 0, & t\in \left(0,T_{k+1}\right),\\
		\hat z_k(T_{k+1})=\hat z_{T,k}, & x\in (0,1).
		\end{cases}
		\end{equation*}
			
			\item $\psi=\chi_\omega \hat z_k$ in $L^1(0,T_k; L^2(0,1))$, where $\hat z_k(\cdot)=e^{A (T_k-\cdot)}z_{T,k}$ and $z_{T,k}=e^{A (T_{k+1}-T_k)}\hat z_{T,k}$.
			
		\end{itemize}
  
		Finally, define
		\begin{equation*}
		\hat z(t)=\hat z_k(t) \mbox{ for } t\in [0,T_k] \mbox{ and  } k\in\mathbb{N}.
		\end{equation*}
		Such  $\hat z$ is well defined over  $[0,T)$ since $\hat z_{k+1}|_{[0,T_k]}=\hat z_k$, which satisfies $\hat z\in C([0,T); L^2(\Omega))\cap L^1(0,T; L^2(\Omega))$, and
		\begin{equation*}
		\begin{cases}
		\partial_t\hat z(x,t)+A \hat z(x,t)=0, &\left( x ,t\right) \in (0,1)\times(0,T),\\ 
  \hat z(1,t)=BC_\alpha\big(\hat z(\cdot,t)\big)
= 0, & t\in \left(0,T\right).
		\end{cases}
		\end{equation*}
		Furthermore,
		\begin{equation*}
		\psi=\chi_\omega \hat z \mbox{ in } L^1(0,T; L^2(0,1)),
		\end{equation*}
		where $\hat z(\cdot)=e^{A (s-\cdot)}z_{T,s}$ with $z_{T,s}\in L^2(0,1)$ and $s\in (0,T)$.
		From this  we obtain $Y_T\subset Z_T$.

		{\it Step 2:  Show that $Z_T\subset Y_T$.}
		
		In fact, by the definition of $Z_T$, for any $\xi\in Z_T$,  there exists $z$ such that
		\begin{equation*}
		\xi=\chi_\omega z\in L^1(0,T; L^2(0,1))
		\end{equation*}
		and for every $s\in (0,T)$ there exists $z_{T,s}\in L^2(0,1)$ such that
		\begin{equation*}
z(\cdot)=e^{A(s-\cdot)}z_{T,s} \mbox{ on } [0,s].
		\end{equation*}
  
		Let $\{T_n\}_{n\geq 1}\subset (\frac{T}{2}, T)$ with $T_n\uparrow T$. Then, for every $n\in\mathbb{N}$, there exists $z_{T,n}\in L^2(0,1)$ such that
		\begin{equation*}
z(\cdot)=e^{A (T_n-\cdot)}z_{T,n} \mbox{ on } [0,T_n].
		\end{equation*}
		This implies that for all $k\geq n$,
		\begin{equation}\label{9.18.16}
		z_{T,n}=z(T_n)=e^{A (T_k-T_n)}z_{T,k}.
		\end{equation}
		For any $l\in\mathbb{N}$, assume that $z_l(\cdot)=z(\cdot; z_{T,l})$, i.e., $z_l$ is the solution to the following system
		\begin{equation*}
		\left\{\!\!
\begin{array}{ll}
z_{l,t}(x,t) +Az_{l}(x,t)= 0, & \left( x ,t\right) \in (0,1)\times(0,T),   \\[2mm]
z_{l}(1,t)=BC_\alpha\big(z_{l}(\cdot,t)\big)
= 0, & t\in \left(0,T\right), \\[3mm]
z_{l}\left(x, T\right) =z_{T,l}, &  x\in (0,1)\,.
\end{array}
\right.
		\end{equation*}
		Then for every  fixed $n\in\mathbb{N}$ and   all $l, k\geq n$ with $k\geq l$, we have
		\begin{equation*}
		\begin{split}
		&\|z_k(T_n)-z_l(T_n)\|\\
		=& \left\|e^{A(T-T_n)}z_{T,k}-e^{A(T-T_n)}z_{T,l}\right\|
		=\left\| e^{A (T-T_k)}e^{A(T_k-T_n)}z_{T,k}-e^{A (T-T_l)}e^{A (T_l-T_n)}z_{T,l}\right\|\\
		=& \left\|e^{A (T-T_k)} z_{T,n}-e^{A (T-T_l)}z_{T,n}\right\|
		=\left\| e^{A(T-T_k)}\left(z_{T,n}-e^{A (T_k-T_l)}z_{T,n}\right)\right\|\\
		\leq&\left\| z_{T,n}-e^{A (T_k-T_l)}z_{T,n}\right\|,
		\end{split}
		\end{equation*}
		where we used \eqref{9.18.16} in  the third equality above and the operator $A$ generates a $C_0$-semigroup of linear
contractions in the last inequality above, see \cite{cannarsa2012unique}.  This
		implies that $\{z_k(T_n)\}_{k\geq n}$ is a Cauchy sequence in  $L^2(0,1)$. Furthermore,
		\begin{equation*}
		\begin{split}
		\|z_k(T_n)-z_{T,n}\|
		=& \left\|e^{A(T-T_n)}z_{T,k}-z_{T,n}\right\|
		= \left\|e^{A(T-T_k)}z_{T,n}-z_{T,n}\right\|
		\rightarrow 0 \mbox{ as } k\rightarrow\infty,
		\end{split}
		\end{equation*}
		from which and $z_{T,n}=z(T_n)$ and similar to the proof of (\ref{9.18.8}) we obtain that
\begin{equation}\label{2.2.1.3}
		\chi_\omega z_k\rightarrow \chi_\omega z \mbox{ strongly in } L^1(0,T_n; L^2(0,1)) \mbox{ as } k\rightarrow \infty \mbox{ for every } n\in\mathbb{N}.
		\end{equation}	
  
  Now, it remains to show that
		\begin{equation}\label{2.2.1.2}
		\chi_\omega z_k \rightarrow \chi_\omega z \mbox{ strongly in } L^1(0,T; L^2(0,1)).
		\end{equation}
		Let $\tilde z$ satisfy
		\begin{equation*}
		\begin{cases}
		\partial_t\tilde z(x,t)+A \tilde z(x,t)=0, &\left( x ,t\right) \in (0,1)\times(-T,T),\\
 \tilde z (1,t)=BC_\alpha\big(\tilde  z(\cdot,t)\big)
 = 0, & t\in \left(-T,T\right),
		\end{cases}
		\end{equation*}
		and
		\begin{equation*}
		\tilde z=z \mbox{ in } (0,1)\times(0,T).
		\end{equation*}
		This is a backward  equation which must admit  solution $\tilde z$, and
		\begin{equation*}
		\tilde z\in C([-T, T); L^2(0,1))\cap L^1(-T, T; L^2(0,1)).
		\end{equation*}
		By the definitions of $\tilde z$ and $z$,  we have
		\begin{equation*}
		z_k(t)=\tilde z(t-(T-T_k)) \mbox{ on } t\in (0,T).
		\end{equation*}
		By the absolutely continuity of $\chi_\omega \tilde z$,  (i.e., for any $\varepsilon>0$,
		there exists $\delta(\varepsilon)>0$ such that for all $a,b\in [-T,T]$ with $a<b$ and $|a-b|\leq \delta(\varepsilon)$) we have
		\begin{equation}\label{7.17.1}
		\|\chi_\omega \tilde z\|_{L^1(a,b; L^2(0,1))}<\varepsilon,
		\end{equation}
		and by the continuity of $\tilde z$, for $\varepsilon>0$ as above, there exists $\eta(\varepsilon)>0$ such that for all $a,b\in [-T, T-\frac{\delta(\varepsilon)}{2}]$ with $a<b$ and $|a-b|<\eta(\varepsilon)$,
		\begin{equation}\label{7.17.2}
		\|\tilde z(a)-\tilde z(b)\|<\varepsilon.
		\end{equation}
  
		Let $k_0=k_0(\varepsilon)$ be such that for all $k\geq k_0$,
		\begin{equation*}
		0<T-T_k\leq \eta(\varepsilon) \mbox{ for all } k\geq k_0.
		\end{equation*}
		This, together with \eqref{2.2.1.3}, \eqref{7.17.1} and \eqref{7.17.2}, implies that
		\begin{equation*}
		\begin{split}
		&\|\chi_\omega z-\chi_\omega z_k\|_{L^1(0,T; L^2(0,1))}\\
		=& \|\chi_\omega \tilde z(\cdot)-\chi_\omega \tilde z(\cdot-(T-T_k))\|_{L^1(0,T; L^2(0,1))}\\
		\leq & \|\chi_\omega \tilde z(\cdot)-\chi_\omega \tilde z(\cdot-(T-T_k))\|_{L^1(0,T-\frac{\delta(\varepsilon)}{2}; L^2(0,1))}
		+\|\chi_\omega \tilde z(\cdot-(T-T_k))\|_{L^1(T-\frac{\delta(\varepsilon)}{2}, T; L^2(0,1))}\\
		&+\|\chi_\omega \tilde z\|_{L^1(T-\frac{\delta(\varepsilon)}{2}, T; L^2(0,1))}\\
		\leq& \|\tilde z(\cdot)-\tilde z(\cdot-(T-T_k))\|_{L^1(0,T-\frac{\delta(\varepsilon)}{2};L^2(0,1))}
		+ \|\chi_\omega \tilde z(\cdot-(T-T_k))\|_{L^1(T-\frac{\delta(\varepsilon)}{2}, T; L^2(0,1))}\\
		&+\|\chi_\omega \tilde z\|_{L^1(T-\frac{\delta(\varepsilon)}{2}, T; L^2(0,1))}\\
		\leq& \left(T-\frac{\delta(\varepsilon)}{2}\right)\varepsilon+2 \varepsilon.
		\end{split}
		\end{equation*}
	  This gives \eqref{2.2.1.2}.
		We have thus proved $Z_T\subset Y_T$. The Lemma \ref{9.18.3} is then proved by combination of  Steps 1 and 2.
	\end{proof}
 We are now ready to prove Theorem \ref{9.15.1}.
\begin{proof}[\textbf{Proof of Theorem~\ref{9.15.1}}]
We carry out this proof by four steps.

{\it Step 1.  The solvability of $V(T,y_0)$.}

Let $z$ with $\chi_\omega z\in Z_T$. From Lemma \ref{9.18.3}, we have
\begin{equation*}
    \begin{split}
        &J(\chi_\omega z)\\
        =& \frac{1}{2}\left\|\chi_\omega z\right\|_{L^1(0,T; L^2(0,1))}^2+\left\langle y_0, z(0)\right\rangle+\int_0^T\int_0^1\sum_{i=1}^2\chi_{\omega_i}u_i^*(g^*)zdxdt\\
        \geq&  \frac{1}{2}\left\|\chi_\omega z\right\|_{L^1(0,T; L^2(0,1))}^2-\|y_0\| \|z(0)\|-\max\{M_1,M_2\}\|\chi_\omega z\|_{L^1(0,T; L^2(0,1))}\\
        \geq& \frac{1}{2}\left\|\chi_\omega z\right\|_{L^1(0,T; L^2(0,1))}^2-C\|y_0\| \left\|\chi_\omega z\right\|_{L^1(0,T; L^2(0,1))}-\max\{M_1,M_2\}\|\chi_\omega z\|_{L^1(0,T; L^2(0,1))},
    \end{split}
\end{equation*}
where we used observability inequality \eqref{eq:obs-inq1} and  $(A_1)$,
		which implies that
\begin{equation*}
		\lim_{\|\chi_\omega z\|_{L^1(0,T; L^2(0,1))}\rightarrow\infty}J(\chi_\omega z)=\infty.
		\end{equation*}
		Since  $J(\chi_\omega z)$ is convex and  continuous, therefore there exists $z_*$ with $\chi_\omega z_*\in Z_T$ such that
		\begin{equation}\label{9.18.23}
		V(T,y_0)=\frac{1}{2}\left\|\chi_\omega z_*\right\|_{L^1(0,T; L^2(0,1))}^2+\left\langle y_0, z_*(0)\right\rangle+\int_0^T\int_0^1\sum_{i=1}^2\chi_{\omega_i}u_i^*(g^*)z_*dxdt.
		\end{equation}
		
{\it Step 2. Construct a control $\bar{g}$.}

Case ($a$). If $z_{*,T}\in L^2(0,1)\setminus\{0\}$,  noting that $z_*(t)=e^{A(T-t)}z_{*,T}$ for all $t\in[0,T]$, from observability inequality (\ref{eq:obs-inq1}), we obtain \begin{equation}\label{***}
		z_*(t)\neq 0 \mbox{ for a.e. } t\in (0,T).
		\end{equation} 
  Hence, for all $h\in\mathbb{R}$, we have
		\begin{align*}
		0\leq&
		\frac{1}{2}\left\|\chi_\omega(z_*+hz)\right\|_{L^1(0,T; L^2(0,1))}^2+\left\langle  y_0,  (z_*+hz)(0)\right\rangle+\int_0^T\int_0^1\sum_{i=1}^2\chi_{\omega_i}u_i^*(g^*)(z_*+hz)dxdt\\
  &-\frac{1}{2}\left\|\chi_\omega z_*\right\|_{L^1(0,T;L^2(0,1))}^2-\left\langle  y_0,  z_*(0)\right\rangle-\int_0^T\int_0^1\sum_{i=1}^2\chi_{\omega_i}u_i^*(g^*)z_*dxdt\\
  =& \frac{1}{2} \int_0^T\frac{2h\int_0^1 \chi_{\omega}z_*(t)\chi_{\omega}z(t) dx+h^2\int_0^1 \chi_\omega z(t)^2 dx}{\|\chi_{\omega}z_*(t)+h\chi_\omega z(t)\|+\|\chi_{\omega}z_*(t)\|} dt\\
  &\cdot \,\,\int_0^T\left(\|\chi_\omega z_*(t)+h\chi_\omega z(t)\|+\|\chi_\omega z_*(t)\|\right) dt\\
		&+h\langle y_0, z(0)\rangle+h\int_0^T\int_0^1\sum_{i=1}^2\chi_{\omega_i}u_i^*(g^*)zdxdt.
		\end{align*}
		Passing to the limit as $h\rightarrow 0$,  we arrive at the Euler-Lagrange equation: for all $z$ with $\chi_\omega z\in Z_T$,
		\begin{equation}\label{EL5}
\int_0^T\frac{\int_0^1 \chi_\omega z_*(t)z(t)  dx}{\|\chi_\omega z_*(t)\|} dt\cdot \int_0^T\|\chi_\omega z_*(t)\| dt+\langle y_0, z(0)\rangle+\int_0^T\int_0^1\sum_{i=1}^2\chi_{\omega_i}u_i^*(g^*)zdxdt=0.
		\end{equation}
		In particular,
		\begin{equation}\label{9.18.18}
		\|\chi_\omega z_*\|_{L^1(0,T; L^2(0,1))}^2+\langle y_0, z_*(0)\rangle+\int_0^T\int_0^1\sum_{i=1}^2\chi_{\omega_i}u_i^*(g^*)z_*dxdt=0.
		\end{equation}
		By \eqref{***}, let
		\begin{equation}\label{9.18.22}
		\bar{g}(t)=\frac{\chi_\omega z_*(t)}{\|\chi_\omega z_*(t)\|}\|\chi_\omega z_*\|_{L^1(0,T; L^2(0,1))} \mbox{ for a.e. } t\in (0,T),
		\end{equation}
		then we have
		\begin{equation}\label{9.18.21}
		\| \bar{g}\|_{L^\infty(0,T;L^2(0,1))}=\|\chi_\omega z_*\|_{L^1(0,T; L^2(0,1))}.
		\end{equation}
Case ($b$). If $z_{*,T}=0$, similar to Case $(a)$,  we arrive at the Euler-Lagrange equation: for all $z$ with $\chi_\omega z\in Z_T$,
		\begin{equation}\label{EL5-2}
\langle y_0, z(0)\rangle+\int_0^T\int_0^1\sum_{i=1}^2\chi_{\omega_i}u_i^*(g^*)zdxdt=0.
		\end{equation}
Let	\begin{equation}\label{9.18.22-2}
		\bar{g}(t)=0 \mbox{ for a.e. } t\in (0,T).
		\end{equation}	
		
	{\it Step 3. Show that $g^*(t)=\frac{\chi_\omega z_*(t)}{\|\chi_\omega z_*(t)\|}\|\chi_\omega z_*\|_{L^1(0,T; L^2(0,1))} $ or $g^*(t)=0$.}

Case $(i)$. If \eqref{9.18.22} holds, then multiplying $z$ with $\chi_\omega z\in X_T\subset Z_T$ on both sides of the first equation of the following system
		\begin{equation*}
		\begin{cases}
\begin{array}{ll}
y_t(x,t) = Ay(x,t) +\chi_\omega \bar{g}+\chi_{\omega_1} u_1^*(g^*)+\chi_{\omega_2} u_2^*(g^*), & \left( x ,t\right) \in (0,1)\times(0,T),   \\[2mm]
y(1,t)=BC_\alpha\big(y(\cdot,t)\big)
= 0, & t\in \left(0,T\right), \\[3mm]
y\left(x, 0\right) =y_{0}(x), &  x\in (0,1)\,,
\end{array}
		\end{cases}
		\end{equation*}
		and integrating with respect to  $x\in(0,1)$ and $t\in(0,T)$, it produces
\begin{align*}
	&\langle y(T; y_0,\bar{g}, u_1^*(g^*),u_2^*(g^*)),z_T \rangle	-\langle y(0), z(0)\rangle-\int_0^T\int_0^1\sum_{i=1}^2\chi_{\omega_i}u_i^*(g^*)zdxdt\\ =&\int_0^T\frac{\int_\omega z_*(t)z(t)  dx}{\|\chi_{\omega} z_*(t)\|} dt\cdot \|\chi_\omega z_*\|_{L^1(0,T; L^2(0,1))},
		\end{align*}
which, along with \eqref{EL5}, it implies that 
 $$
 y(T; y_0,\bar{g}, u_1^*(g^*),u_2^*(g^*))=0.
 $$
From the proof of Lemma~\ref{lemma-g-exist}, we know that $u_1^*(\bar{g})=u_1^*(g^*)$ or $u_2^*(\bar{g})=u_2^*(g^*)$.\begin{itemize}
  \item [Case 1.] If $u_1^*(\bar{g})=u_1^*(g^*)$, 
  which, along with Remark \ref{remark-N1-N2} and $y(T;y_0, \bar g, u_1^*(\bar g_1), u_2^*(g^*))=0$, implies that
$u_2^*(\bar{g})=u_2^*(g^*)$.
 \item [Case 2.] If $u_2^*(\bar{g})=u_2^*(g^*)$, similar to Case 1, we also have $u_1^*(\bar{g})=u_1^*(g^*)$.
\end{itemize}
From these discussions, we obtain $u_1^*(\bar{g})=u_1^*(g^*)$ and $u_2^*(\bar{g})=u_2^*(g^*)$.

Next,	since $g^*$ is the optimal control of the problem {\bf (P2)}, then we have
		\begin{equation*}
		y(T; y_0,g^*, u_1^*(g^*),u_2^*(g^*))=0.
		\end{equation*}
Since $Z_T=Y_T$, similar to the proof of (\ref{9.18.8}), assume $\{z_{T,n}\}_{n\geq 1}\subset L^2(0,1)$ be such that
	\begin{equation}\label{1}
		\chi_\omega z(\cdot; z_{T,n})\rightarrow \chi_\omega z_* \mbox{ strongly in } L^1(0,T; L^2(0,1)),
		\end{equation}
  and
  \begin{equation}\label{11}
		\chi_{\omega_i} z(\cdot; z_{T,n})\rightarrow \chi_{\omega_i} z_* \mbox{ strongly in } L^1(0,T; L^2(0,1)),~\text{for}~i=1,2.
		\end{equation}
Multiplying $z(\cdot; z_{T,n})$ on both sides of the first equation of the following system
		\begin{equation*}
		\begin{cases}
\begin{array}{ll}
y_t(x,t) = Ay(x,t) +\chi_\omega g^*+\chi_{\omega_1} u_1^*(g^*)+\chi_{\omega_2} u_2^*(g^*), & \left( x ,t\right) \in (0,1)\times(0,T),   \\[2mm]
y(1,t)=BC_\alpha\big(y(\cdot,t)\big)
= 0, & t\in \left(0,T\right), \\[3mm]
y\left(x, 0\right) =y_{0}(x), &  x\in (0,1)\,,
\end{array}
		\end{cases}
		\end{equation*}
		and integrating with respect to  $x\in(0,1)$ and $t\in(0,T)$, it produces
\begin{equation}\label{****}
		-\langle y(0), z(0;z_{T,n})\rangle-\int_0^T\int_0^1\sum_{i=1}^2\chi_{\omega_i}u_i^*(g^*)z(t;z_{T,n})dxdt=\int_0^T \langle g^*(t), \chi_\omega z(t;z_{T,n})\rangle dt.
		\end{equation}
By the same arguments as  the proof of Lemma \ref{9.18.3}, there exists a subsequence of $\{z_{T,n}\}_{n\geq 1}$, still denoted by itself, such that		\begin{equation}\label{aa}
		z(0; z_{T,n})\rightarrow z_*(0) \mbox{ weakly in } L^2(0,1).
		\end{equation}
By \eqref{1}-\eqref{aa}, letting $n\rightarrow\infty$, we obtain that
		\begin{equation*}
		-\langle y(0), z_*(0)\rangle-\int_0^T\int_0^1\sum_{i=1}^2\chi_{\omega_i}u_i^*(g^*)z_*(t)dxdt=\int_0^T \langle g^*(t),\chi_\omega z_*(t)\rangle dt,
		\end{equation*}
which, along with \eqref{9.18.18} and \eqref{9.18.21}, we have
		\begin{equation}\label{9}
		\| \bar{g}\|_{L^\infty(0,T;L^2(0,1))}=\|\chi_\omega z_*\|_{L^1(0,T; L^2(0,1))}\leq \| g^*\|_{L^\infty(0,T; L^2(0,1))}.
		\end{equation}
This implies that
$$
\|\bar{g}\|_{L^\infty(0,T;L^2(0,1))}=\| g^*\|_{L^\infty(0,T; L^2(0,1))}=N(T,y_0).
$$
Therefore, $\bar{g}$ is also an optimal control of problem {\bf (P2)}. Replacing  $g^*$ by $\bar{g}$ and together with \eqref{9.18.22}, we have
	$$
g^*(t) =\frac{\chi_\omega z_*(t)}{\|\chi_\omega z_*(t)\|}\|\chi_\omega z_*\|_{L^1(0,T; L^2(0,1))} \mbox{ for a.e. } t\in (0,T).
 $$	
 
Case $(ii)$. If \eqref{9.18.22-2} holds, then multiplying $z$ with $\chi_\omega z\in X_T\subset Z_T$ on both sides of the first equation of the following system
		\begin{equation*}
		\begin{cases}
\begin{array}{ll}
y_t(x,t) = Ay(x,t) +\chi_{\omega_1} u_1^*(g^*)+\chi_{\omega_2} u_2^*(g^*), & \left( x ,t\right) \in (0,1)\times(0,T),   \\[2mm]
y(1,t)=BC_\alpha\big(y(\cdot,t)\big)
 = 0, & t\in \left(0,T\right), \\[3mm]
y\left(x, 0\right) =y_{0}(x), &  x\in (0,1)\,,
\end{array}
		\end{cases}
		\end{equation*}
		and integrating with respect to  $x\in(0,1)$ and $t\in(0,T)$, it produces
\begin{align*}
	\langle y(T; y_0,0, u_1^*(g^*),u_2^*(g^*)),z_T \rangle	-\langle y(0), z(0)\rangle-\int_0^T\int_0^1\sum_{i=1}^2\chi_{\omega_i}u_i^*(g^*)zdxdt =0,
		\end{align*}
which, along with \eqref{EL5-2}, it implies that 
 $$
 y(T; y_0,0, u_1^*(g^*),u_2^*(g^*))=0.
 $$
Similar to the proof of Case ($i$), we obtain $u_1^*(0)=u_1^*(g^*)$ and $u_2^*(0)=u_2^*(g^*)$. This shows that  $$
 y(T; y_0,0, u_1^*(0),u_2^*(0))=0.
 $$
Therefore, by the definition of $N(T,y_0)$, which defined in \eqref{9.11.2},  $\bar{g}=0$ is also an optimal control of problem {\bf (P2)}. Replacing  $g^*$ by $\bar{g}$ and together with \eqref{9.18.22-2}, we have
	$$
g^*(t) =0 \mbox{ for a.e. } t\in (0,T).
 $$

{\it Step 4. Conclude that $V(T,y_0)=-\frac{1}{2}N(T,y_0)^2$.} 	

		By Step 3, it shows that\\
Case ($I$). If
   $$
   g^*(t) =\frac{\chi_\omega z_*(t)}{\|\chi_\omega z_*(t)\|}\|\chi_\omega z_*\|_{L^1(0,T; L^2(0,1))} \mbox{ for a.e. } t\in (0,T),
   $$
   then
   \begin{equation*}
		N(T,y_0)=\|\chi_\omega z_*\|_{L^1(0,T; L^2(0,1))},
		\end{equation*}
		which, together with \eqref{9.18.18} and \eqref{9.18.23}, deduces that
		\begin{equation*}
		V(T,y_0)=-\frac{1}{2} \|\chi_\omega z_*\|_{L^1(0,T;L^2(0,1))}^2=-\frac{1}{2}N(T,y_0)^2.
		\end{equation*} 
Case ($II$). If
   $$
   g^*(t)=0,\,\mbox{ for a.e. } t\in (0,T),
   $$
   then
    $$
  V(T,y_0)=N(T,y_0).
  $$
From these, we have
$$
V(T,y_0)=-\frac{1}{2}N(T,y_0)^2.
$$
The proof is completed.
	\end{proof}

\vspace{3mm}

\noindent{\bf Acknowledgement}

\vspace{2mm}

The first three authors are supported by the National Natural Science Foundation of China under grant 11871478, the Science Technology Foundation of Hunan Province.

The last author is supported by the National Natural Science Foundation of China under grant 11971363, and by the Fundamental Research Funds for the Central Universities under grant 2042023kf0193.

\bibliographystyle{abbrvnat}
\bibliography{ref.bib}

\end{document}